\documentclass[12pt]{amsart}
\usepackage{amsmath}
\usepackage{amssymb}
\usepackage[all]{xy}
\usepackage{longtable}
\usepackage{euscript}
\usepackage[scaled]{helvet} 

\newtheorem{theorem}[subsection]{Theorem}
\newtheorem{lemma}[subsection]{Lemma}
\newtheorem{proposition}[subsection]{Proposition}

\newtheorem{conjecture}[subsection]{Conjecture}

\theoremstyle{definition}
\newtheorem{definition}[subsection]{Definition}

\theoremstyle{remark}
\newtheorem{remark}[equation]{Remark}
\newenvironment{proof1}{\noindent {\it Proof of Theorem}}{\hfill   $\Box$\vspace{0.3cm}}

\numberwithin{equation}{subsection}

\newcommand{\FF}{\mathbb{F}}
\newcommand{\ZZ}{\mathbb{Z}}
\newcommand{\QQ}{\mathbb{Q}}

\newcommand{\CC}{\mathbb{C}}

\newcommand{\NN}{\mathbb{N}}

\newcommand{\laurent}[2]{{#1 (\!( #2 )\!)}}

\usepackage{color}

\begin{document}
\title[Anderson-Thakur polynomials and multizeta values]{{\small Anderson-Thakur polynomials and multizeta values in finite characteristic}}
\author{Huei-Jeng Chen}
\date{}
\maketitle
\section{Introduction}
The study of arithmetic of zeta values begins by Euler's famous evaluations: for $m\in \NN$,
\[
\zeta(2m)=\frac{-B_{2m}\left( 2\pi\sqrt{-1} \right)^{2m} }{2(2m)!},
\]
where $B_{2m}\in \QQ$ are Bernoulli numbers. Euler's formula implies that $\zeta(n)/(2\pi \sqrt{-1})^{n}$ is rational if and only if $n$ is even. As generalizations of zeta values, Euler studied multiple zeta values $\zeta(s_{1},\cdots,s_{r})$, where $s_{1},\ldots,s_{r}$ are positive integers with $s_{1}\geq 2$. Although there exist simple relations between zeta and multiple zeta values, such as $\zeta(2,1) = \zeta (3)$, sorting out all relations among these multiple zeta values is a much involved problem. Here $r$ is called the depth and $w:=\sum_{i=1}^{r}s_{i}$ is called the weight of $\zeta(s_{1},\ldots,s_{r})$. We call $\zeta(s_{1},\ldots,s_{r})$ \emph{Eulerian} if the ratio $\zeta (s_{1},\ldots,s_{r}) / \bigl( 2\pi \sqrt{-1})^w$
is rational.

Carlitz introduced and derived an analogue of Euler's formula for what we now called Carlitz zeta values $\zeta_{A} (n)$. Let $A=\FF_q[\theta]$ be the polynomial ring in the variable $\theta$ over a finite field $\FF_q$ and $K = \FF_q(\theta)$ be its quotient field. Let $\bf{C}$ be the Carlitz module and $\widetilde{\pi}$ is a fundamental period of $\bf{C}$. The Carlitz exponential
function is defined by $\exp_{\bf{C}}(z) = \sum_{n\geq 0} \displaystyle \frac{z^{q^n}}{D_n}$. We denote by $\Gamma_{n+1} \in A$ the Carlitz factorials and
$BC(n)\in K$ by the Bernoulli-Carlitz numbers. Carlitz showed that $$\zeta_A(n) := \sum_{a\in A+} \displaystyle \frac{1}{a^n}=\displaystyle\frac{BC(n)}{\Gamma_{n+1}}\widetilde{\pi}^n$$ if $q-1 | n.$ Carlitz's result implies that $\zeta_A(n) / \tilde{\pi}^n$ is rational in $K$ if and only if $q-1 | n$.

Anderson and Thakur \cite{AT} related the interesting value $\zeta_{A}(n)$ to a special integral point $Z_{n}$ on $\mathbf{C}^{\otimes n}(A)$ via the logarithm map of $\mathbf{C}^{\otimes n}$, where $\mathbf{C}^{\otimes n}$ denotes the $n$-th tensor power of the Carlitz module. As a consequence, one has that $\zeta_{A}(n)/\tilde{\pi}^{n}$ is rational if and only if $Z_{n}$ is a $\FF_{q}[t]$-torsion point, and this condition is equivalent to $n$ being divisible by $q-1$. In \cite{AT} a key role is played by a sequence of distinguished polynomials $H_n\in A[t]$, now called the Anderson-Thakur polynomials. On the other hand, Yu \cite{Yu} also showed that the transcendence of $\zeta_{A}(n)/\tilde{\pi}^{n}$ over $K$ is equivalent to $Z_{n}$ being non-torsion on $\mathbf{C}^{\otimes n}(A)$, whence deriving that $\zeta_{A}(n)/\tilde{\pi}^{n}$ is algebraic over $K$ if and only if $\zeta_{A}(n)/\tilde{\pi}^{n}$ is rational in $K$.

In the last decade, Thakur \cite{Tbook, T1} initiated the study of multizeta values $\zeta_A(s_{1},\cdots,s_{r})$. He and his co-workers discovered interesting relations among some multizeta values.
Call $\zeta_A(s_{1},\ldots,s_{r})$ \emph{Eulerian} (\emph{zeta-like} resp.) if the ratio $\zeta (s_{1},\cdots,s_{r}) / \tilde{\pi}^w$
( $\zeta_A (s_{1},\ldots,s_{r}) / \zeta_A (w)$ resp.) is rational in $K$. A basic question in this respect is to find all Eulerian/zeta-like multizeta values. In \cite{LRT}, Lara Rodriguez and Thakur gave particularly precise formulas for certain families of Eulerian/zeta-like multizeta values and conjectured other ones. Their conjectures are supported by numerical data from continued fraction computations. On the other hand, Chang \cite{Cha} also proved the subtle fact that these ratios $\zeta (s_{1},\cdots,s_{r}) / \tilde{\pi}^w$,
$\zeta_A (s_{1},\ldots,s_{r}) / \zeta_A (w)$ are either rational or transcendental over $K$.

In an effort to understand relations among multizeta values, Chang, Papanikolas and Yu \cite{CPY} established an effective criterion for Eulerian/zeta-like multizeta values by constructing an abelian t-module $E'$ defined over $A$ and relating the values $\zeta_A(s_{1},\cdots,s_{r})$, $\zeta_A(w)$ to specific integral points $\mathbf{v}_{\mathfrak{s}}$, $\mathbf{u}_{\mathfrak{s}}$ on $E'(A)$. They proved that $\zeta_A(s_{1},\cdots,s_{r})$ is Eulerian (zeta-like) if and only if $\mathbf{v}_{\mathfrak{s}}$ is a $\FF_q[t]$- torsion point(respectively, $\mathbf{u}_{\mathfrak{s}}$ and $\mathbf{v}_{\mathfrak{s}}$ have an $\FF_q[t]$-linear relation inside $E'(A)$). The integral points $\mathbf{v}_{\mathfrak{s}}$, $\mathbf{u}_{\mathfrak{s}}$ are constructed using the Anderson-Thakur polynomials. Their theory connects possible $\FF_q(\theta)$-linear relation of $\zeta_A(s_{1},\cdots,s_{r})$ and $\zeta_A(w)$ explicitly with the possible $\FF_q[t]$-linear relation among $\mathbf{v}_{\mathfrak{s}}$ and $\mathbf{u}_{\mathfrak{s}}$ inside $E'(A)$.

Just recently, Kuan-Lin \cite{LK} implemented algorithms basing on the criterion of Chang-Papanikolas-Yu. They have collected more extensive data on zeta-like and Eulerian multizeta values over the polynomial rings $\FF_q[\theta]$. Particularly in \cite{CPY,LRT}, a conjectured rule is spelled out to specify all Eulerian multizeta values. Lists given in \cite{LK} suggest more families of zeta-like multizeta values of arbitrary depth. These families are not covered by \cite{LRT}. It is observed that there should be only a few zeta-like families in higher depth, because of the conjectured "splicing" condition (cf. \cite{LRT}). Finding all zeta-like multizeta values is now in sight.

Inspired by this development we study Anderson-Thakur polynomials in more details in this paper, for the purpose of deriving exact rational ratio between $\zeta_A(s_{1},\cdots,s_{r})$ and $\zeta_A(w)$ whenever such a ratio exists. In particular, we are able to verify : (1) Conjecture 4.6 of \cite{LRT}, (2) Conjecture 5 of \cite{LK}, (3) the conjectured list of all Eulerian multizeta values given in \cite{CPY}, Section 6.2, are indeed Eulerian. The strategy for proving zeta-like property for given multizeta values is to handle recurrence relations among Anderson-Thakur polynomials $H_n$. In view of  the fact that these $H_n$ are polynomials in both $\theta$ and $t$ over $\FF_q$, we use Lucas Theorem to establish $q$-th power recurrence when $n$ has particular $q$-adic \lq\lq shape\rq\rq. Combining with the obvious linear recurrence relating $H_n$ to $H_{n-q^i}$, we eventually arrive at more transparent formulas for $H_n$.

The contents of this paper are arranged as follows.  In Section \ref{NotationandDefintion}, we set up preliminaries and introduce the conjectured families of zeta-like multizeta values given in \cite{LK} and \cite{LRT}, which we will prove later. In Section \ref{zeta-likeAT}, we use generalized Lucas Theorem \cite[p.75-76]{Lucas} to study Anderson-Thakur polynomials. Then in Section \ref{zeta-like} we apply  Chang-Papanikolas-Yu's theorem \cite[Theorem 2.5.2]{CPY} to verify that all previously conjectured families of zeta-like multizeta values are indeed zeta-like with exact formulas given in Theorem \ref{mainthm3}. At the end of this paper we provide `recursive' relations for two very special families of multizeta values and derive that they are Eulerian  (Theorem \ref{mainthm1}, \ref{mainthm2}) in Section \ref{Eulerian}.

\section{Preliminaries for Multizeta values}\label{NotationandDefintion}
\subsection{{\small Notation}} We adopt the notation below in the following chapters.
\begin{longtable}{p{0.2truein}@{\hspace{2pt}$=$\hspace{4pt}}p{4truein}}
$\FF_q$& a finite field with $q=p^m$ elements.\\
$K$& $\FF_q(\theta)$, the rational function field in the variable $\theta$.\\
$\infty$ & $1/\theta$, the infinite place of $K$.\\
$|\ \ |$ & the nonarchimedean absolute value on $K$ corresponding to $\infty$.\\
$K_{\infty}$ & $\FF_q((1/\theta))$, the completion of $K$ with respect to the absolute value $|\cdot|$.\\
$\CC_\infty$ & the completion of $\overline{K_\infty}$ with respect to
the canonical extension of $\infty$.\\
$A$ & $\FF_q[\theta]$, the ring of polynomials in the variable $\theta$.\\
$A_{+}$ & the set of monic polynomials in $A$.\\
$A_{d}$ & the set of polynomials in $A$ of degree $d$.\\
$A_{d+}$ & $A_{d} \cap A_{+}$,the set of monic polynomials in $A$ of degree $d$.\\
$[n]$ & $\theta^{q^n}-\theta$.\\
$D_n$ & $\prod\limits_{i=0}^{n-1} \theta^{q^n}-\theta^{q^i} = [n][n-1]^q\cdots[1]^{q^{n-1}}$.\\
$L_n$ & $\prod\limits_{i=1}^{n} \theta^{q^i}-\theta = [n][n-1]\cdots[1]$.\\
$t$& a variable independent of $\theta$.
\end{longtable}
\subsection{{\small Multizeta values}}
For $s\in \ZZ$ and $d\in \ZZ_{\geq 0}$, put $$S_d(s) = \sum_{a\in A_{d+}} \frac{1}{a^s} \in K.$$

For a given tuple $(s_1,\cdots,s_r) \in \NN^r$ and $d\in \ZZ_{\geq 0}$, put
$$S_d(s_1,\cdots,s_r) = S_d (s_1) \sum_{d>d_2>\cdots>d_r\geq 0} S_{d_2} (s_2)\cdots S_{d_r} (s_r) \in K.$$

For $k\in \ZZ$ the Carlitz-Goss zeta values are defined by
$$\zeta (k) = \sum_{d=0}^{\infty} S_d(k) = \sum_{a\in A_{+}} \frac{1}{a^k} \in K_{\infty}.$$
For a given tuple $(s_1,\cdots,s_r) \in \NN^r$, the Thakur multizeta values of depth $r$ and weight $w=\sum s_i$ are defined by
$$\zeta (s_1,\cdots,s_r) = \sum_{d_1>\cdots>d_r\geq 0} S_{d_1}(s_1)\cdots S_{d_r}(s_r) = \sum_{a_i\in A_{+} \atop  \deg a_1 >\cdots >\deg a_r \geq 0} \frac{1}{a_1^{s_1}\cdots a_r^{s_r}}.$$

\subsection{{\small  Bernoulli-Carlitz numbers $BC(n)$}}
For a non-negative integer $n$, we express $n$ as
$$n=\sum_{i=0}^{\infty} n_{i}q^{i} \quad \textnormal{($0\leq n_{i}\leq q-1$, $n_{i}=0$ for $i\gg 0$)},$$
and we recall the definition of the arithmetic $\Gamma$-function,
$$
\Gamma_{n+1}:=\prod_{i=0}^{\infty} D_{i}^{n_{i}}\in A.
$$
Let $\bf{C}$ be the Carlitz module and $\displaystyle\tilde{\pi} = (-\theta)^{\frac{q}{q-1}}\prod_{i=1}^{\infty} (1-\frac{\theta}{\theta^{q^i}})^{-1}$ be a fundamental period of $\bf{C}$. The Carlitz exponential
function is defined by $\exp_{\bf{C}}(z) = \sum_{n\geq 0} \displaystyle \frac{z^{q^n}}{D_n}$. The Bernoulli-Carlitz numbers $BC(n)\in K$ defined by
$$\displaystyle\frac{z}{\exp_{\bf{C}}(z)} = \sum_{n\geq 0} \displaystyle\frac{BC(n)}{\Gamma_{n+1}}z^n.$$
\\
When $n$ is `even' i.e., $q-1 | n$, Carlitz derived an analogue of Euler's formula as follows:
\begin{lemma}\label{bernoulli} $($Carlitz$)$
\\
$($a$)$ $\zeta_A(n) = \displaystyle\frac{BC(n)}{\Gamma_{n+1}}\widetilde{\pi}^n$ if $q-1 | n$.
\\
$($b$)$ $BC(q^n-q^i) = \displaystyle\frac{(-1)^{n-i}\Gamma_{q^n-q^i+1}}{L_{n-i}^{q^i}}$.

\end{lemma}
\
\\
Combining ($a$), ($b$) we get $\zeta_A(q^n-1) = \displaystyle\frac{(-1)^n}{L_n}\widetilde{\pi}^{q^n-1}$.

\subsection{{\small Anderson-Thakur polynomials}}
First we define polynomials $G_i\in \FF_q [t,\theta]$ for $i\in \NN\cup \{0\}$. Put $G_0=1$.
For $i\in \NN$, let
$$G_i= \prod\limits_{j=1}^{i} (t^{q^i}-\theta^{q^j}). $$
For $n=0,1,2,\ldots$ we define the sequence of Anderson-Thakur polynomials $H_{n}\in A[t]$ by the generating function identity
\[
\left( 1-\sum_{i=0}^{\infty} \frac{ G_{i} }{ D_{i}|_{\theta=t}} x^{q^{i}}  \right)^{-1}=\sum_{n=0}^{\infty} \frac{H_{n}}{\Gamma_{n+1}|_{\theta=t}} x^{n}.
\]
We note that for $0\leq n\leq q-1$ we have $H_{n}=1$. The following two identities follows from the above definition:
\\
\\\
$\begin{aligned}
&(a)\ \sum_{n=0}^{\infty}  \frac{H_{n}}{\Gamma_{n+1}|_{\theta=t}} x^{n} = \sum_{m\geq 0}(\sum_{i=0}^{\infty} \frac{ G_{i} }{ D_{i}|_{\theta=t}} x^{q^{i}})^m,\\
&(b)\ (1-\sum_{i=0}^{\infty} \frac{ G_{i} }{ D_{i}|_{\theta=t}} x^{q^{i}})  (\sum_{n=0}^{\infty}  \frac{H_{n}}{\Gamma_{n+1}|_{\theta=t}} x^{n})=1,
\end{aligned}$
\\
\\
For any infinite vector $\underline{a}=( a_0, a_1, a_2, \cdots )$ with integers $a_i\geq 0$ and $a_j=0$ for $j\gg 0$, put
$m(\underline{a})$:= last index $i$ such that $a_i \ne 0$.
We define $C_{\underline{a}} := \displaystyle\frac{(a_0+\cdots+a_{m(\underline{a})})!}{a_0!\cdots a_{m(\underline{a})}!}$.
For $n\in \NN$, a $q$ power weighted partition is an infinite vector $\underline{a}$ satisfying $n=\sum_{i=0}^{\infty} a_i q^i$.
We have the following lemma giving two ways for explicitly writing Anderson--Thakur polynomials:
\\
\begin{lemma}\label{inductionhn}
\
\\
$($a$)$
For $n\in\NN$, let $S_n = \{\underline{a}\ |\ n= \sum a_iq^i,\ C_{\underline{a}}\not\equiv 0 \mod p\}$ denote the set of all possible $q$ power weighted partition of $n$ with nonzero
$C_{\underline{a}} \mod p$.
Then $$\frac{H_{n}}{\Gamma_{n+1}|_{\theta=t}} = \sum_{\underline{a}\in S_n} C_{\underline{a}} \prod_{i=0}^{\infty} (\frac{G_i}{D_i|_{\theta=t}})^{a_i}.$$
\\
$($b$)$ For $n\in\NN$, $$\frac{H_{n}}{\Gamma_{n+1}|_{\theta=t}} = \sum_{i=0}^{[log_q n]} \frac{G_i}{D_i|_{\theta=t}}\frac{H_{n-q^i}}{\Gamma_{n-q^i+1}|_{\theta=t}}.$$
\end{lemma}
\
\\
We will discuss more details about Anderson-Thakur polynomials in Section \ref{zeta-likeAT}.

\subsection{Revisiting Lucas Theorem}
To compute $C_{\underline{a}}$$\mod p$, a useful tool is a generalization of the Lucas Theorem.
\begin{theorem} \label{Lucas}$($ Dickson \cite{Lucas} $)$
For any infinite vector $\underline{a}$, $C_{\underline{a}}\not\equiv 0 \mod p$ if and only if
there is no carrying in computing the sum $a_1+\cdots+a_{m(\underline{a})}$ in terms of base $p$ expansion. Furthermore, if
$\sum a_i = \sum_{j=0}^{m} n_jp^j$, $a_i = \sum_{j=0}^{m} n_{i,j}p^j$. with $0\leq n_{j}, n_{i,j}\leq p-1$. Then  $C_{\underline{a}} \equiv \prod C_{\underline{n_j}}$ in $\FF_p$, where $\underline{n_j} = ( n_{0,j},\cdots,n_{m(\underline{a}),j},0,0,\cdots )$.
\end{theorem}
\begin{proof}
See \cite[p.75-76]{Lucas}.
\end{proof}
\
\\
By Theorem \ref{Lucas} we see that $C_{\underline{a}}$$\mod p$ can be computed as digits in base $p$ expansion separately. So we try to descend $H_n$ via the maps below. For simplicity we view $C_{\underline{a}}$ as elements in $\FF_p$.

\begin{definition} For any infinite vector $\underline{a}$ with $C_{\underline{a}}\ne 0$, let $\tilde{\underline{a}}= (\tilde{a_0}, \tilde{a_1}, \cdots )$, where $a_i \equiv \widetilde{a_i} \mod q$ with $0 \leq \widetilde{a_i}\leq q-1$.  We define the following `reduction map' of vectors.
\\
$$\tilde{\tilde{\underline{a}}}:= (\frac{a_0-\tilde{a_0}}{q}, \frac{a_1-\tilde{a_1}}{q}, \cdots ).$$
\end{definition}
\
\\
By Theorem \ref{Lucas} we see that
\begin{lemma}
$C_{\underline{a}} = C_{\tilde{\underline{a}}}C_{\tilde{\tilde{\underline{a}}}}$.
\end{lemma}
\
\\
\subsection{{\small Binomial series to the Carlitz module}}\label{binomialseries}
For $k\in\ZZ_{\geq 0}$, let $\Psi_k(u)$ be the polynomials in $K[u]$ defined by $$\exp_C(u\log_C(x))= \sum_{k\geq 0} \Psi_k(u) x^{q^k}.$$ The polynomials $\Psi_k(x)$ are analogues to the classical binomial series $$\left(
                                                                                                                   \begin{array}{c}
                                                                                                                     x \\
                                                                                                                     n \\
                                                                                                                   \end{array}
                                                                                                                 \right)
\in \QQ[x]$$ to the multiplicative group, which are defined by
$$(1+t)^x = \exp(x\log (1+t)) = \sum_{n=0}^{\infty}\left(
                                                                                                                   \begin{array}{c}
                                                                                                                     x \\
                                                                                                                     n \\
                                                                                                                   \end{array}
                                                                                                                 \right) t^n.$$

\begin{proposition}\label{ATbinomial}$($Anderson-Thakur \cite{AT}$)$
$$\Psi_k(x)= \sum_{i=0}^{k} \frac{\prod_{j=1}^{i}(\theta^{q^i}-\theta^{q^{k+j}})}{D_i}(\frac{x}{(-1)^kL_k})^{q^i}.$$ Moreover,
$\Psi_k(a)= 0$ for all $a\in \FF_q[\theta]$ with $\deg_{\theta} a < k$ and $\Psi_k(\theta^k)= 1$.
\end{proposition}
\
\\
This result is another key tool in the proof of Theorem \ref{mainproposition}. For our purpose, we replace $\theta$ by $t$ in Anderson-Thakur's result so that $\Psi_k(x)|_{\theta=t} \in \FF_q(t)[x]$.
\\
\subsection{{\small Conjectures on Eulerian/Zeta-like Multizeta Values}}
There are families of zeta-like multizeta values of arbitrary depth, for instance, in \cite{LRT}, they showed that
for any $q$,
$\zeta_A(1,q-1,(q-1)q, \cdots, (q-1)q^n )$ is zeta-like by giving the ratio of it to $\zeta_A(q^{n+1})$.
There are certainly more families of zeta-like multizeta values of arbitrary depth, the following conjecture is given by \cite{LRT}:

\begin{conjecture}\ \label{LRTconj}
\\
\\
$($a$)$ For any $q,n \geq 1$ and $r\geq 2$,  $$\zeta_{A}(q^n-1, (q-1)q^n, \cdots, (q-1)q^{n+r-2} )
=\displaystyle\frac{[n+r-2][n+r-3]\cdots[n]}{[1]^{q^{n+r-2}}[2]^{q^{n+r-3}}\cdots [r-1]^{q^{n}}}\zeta_{A}(q^{n+r-1}-1).$$
\\
$($b$)$ For any $q,n \geq 0$, $$\zeta_{A}(1, q^2-1, (q-1)q^2, \cdots, (q-1)q^{n+1} )
=\displaystyle\frac{[n+2]-1}{l_1[n+2]}\frac{1}{l_1^{(q-1)q^{n}}l_2^{(q-1)q^{n-1}}\cdots l_{n-1}^{(q-1)q^{2}}l_n^{q^2}}\zeta_{A}(q^{n+2}),$$ where
$l_i:=(-1)^iL_i$.
\\
\\
$($c$)$ For $q>2$, $n\geq 0$ and $r\geq 2$,
$$\zeta_A((q-1)q^n-1, (q-1)q^{n+1},\cdots,(q-1)q^{n+r-1}) = \displaystyle\frac{(-1)^{r+1}[n+r-1][n+r-2]\cdots [n+1]}{[1]^{(q^{r-1}-1)q^n} [2]^{(q^{r-2}-1)q^n} \cdots [r-1]^{(q-1)q^n}}\zeta_A(q^{n+r}-q^n-1).$$
\end{conjecture}
\
\\
\begin{remark} In \cite{LRT} Conjecture \ref{LRTconj} is proved in the depth $2$ case. We refer to
\cite{LRT} for more details, in particular \cite[Theorem $3.1$]{LRT}, where many depth $2$ zeta-like multizeta values are given with precise ratio to $\zeta_A(w)$.
\end{remark}

Basing on Chang-Papanikolas-Yu criterion for Eulerian/zeta-like in \cite{CPY}, Kuan-Lin \cite{LK} wrote an algorithm and tested multizeta values with bounded weights and depths by computer. From their output data, they gave another more extensive conjecture about zeta-like families of arbitrary depth and also specific depth $3$ zeta-like multizeta values.
\\
\begin{conjecture}$($ Kuan-Lin-Yu $)\label{LKconj}$ Suppose that $q>2$. Then we have the following families of zeta-like multizeta values:
\\
\\
$($a$)$ For $q=p^l>2$, $1\leq p^m\leq q$, $n>0$ and $r\geq 2$, consider $N_i\in \NN \cup \{0\}$ for $0\leq i \leq n-1$ such that  $1\leq \sum N_i \leq q-1$.
 If $(q-1)(q^n-\sum N_iq^i)\leq p^m(q-1)q^{n-1}$, then
$\zeta_A(q^n-\sum N_iq^i,p^m(q-1)q^{n-1},\cdots, p^m(q-1)q^{n+r-3})$ is zeta-like. In particular $\zeta_A(1,p^m(q-1),p^mq(q-1),\cdots,p^mq^{r-2}(q-1))$ is zeta-like.
\\
\\
$($b$)$ In the case of depth $r=3$,
$$\zeta_A(1,q(q-1),q^3-q^2+q-1)= \frac{[3]-1}{[3][2][1]^{q^2-q-1}}\zeta_A(q^3).$$
\end{conjecture}
\
\\
\begin{remark}
Conjecture \ref{LRTconj} $($c$)$ is a special case of Conjecture \ref{LKconj}$($b$)$ by taking $p^m=q$, $N_0=N_{n}=1$ and $N_i=0$ for $0<i<n$.
\end{remark}

Note that when the weight $w$ is 'even', the statement that $\zeta_A(s_1, \cdots, s_r)$ is zeta-like is equivalent to that it is Eulerian. In Section \ref{zeta-like} we will prove non-Eulerian part of Conjecture \ref{LRTconj} and  Conjecture \ref{LKconj}. The Eulerian part of Conjecture \ref{LRTconj} will be treated in Section \ref{Eulerian}

\section{{Investigation into Anderson-Thakur polynomials}}\label{zeta-likeAT}
In general Anderson-Thakur polynomials $H_n$ are complicated to investigate. However, for index $n$ having very special $q$-adic expansion, we can
give a nicer and simpler formula for such $H_n$. For example, to prove Conjecture \ref{LKconj} $(b)$, we need to compute the corresponding Anderson-Thakur polynomials $H_0$, $H_{q^2-q-1}$, $H_{q^3-q^2+q-2}$ and $H_{q^3-1}$.
It is known that $H_0=1$. On the other hand, it can be directly proved that $H_{q^2-q-1}=\Gamma_{q^2-q}$, $H_{q^3-1}=\Gamma_{q^3}$ and
$H_{q^2-q}=\Gamma_{q^2-q+1}|_{\theta=t}\displaystyle\frac{(t-\theta^q)^{q-1}}{L_1^{q-1}|_{\theta=t}}$. ( These are special cases of Theorem \ref{mainproposition} ) For $m\in \NN$, let $S_m$ denotes the set of all $q$ power weighted partition $\underline{a}$ of $m$ with $C_{\underline{a}}\ne 0$. Furthermore, for given $\tilde{\underline{a}}$ with $0\leq \tilde{a_i} \leq q-1$, let $S_{m,\tilde{\underline{a}}}$ be the subset of $S_m$ collecting $\underline{a}$ satisfying
$a_i \equiv \tilde{a_i} \mod q$ for all $i$. By Lemma \ref{inductionhn} and using the reduction maps $\tilde{\tilde{\underline{a}}}$, we can compute $H_{q^3-q^2+q-2}$.

\begin{proposition}\label{example}
$$H_{q^3-q^2+q-2}=-[2]^{q-2}\{(t-\theta^q)^{q^2-q+1}+[1]^{q^2-1}(t-\theta^q)\}.$$
\end{proposition}
\begin{proof}
For any $q$ power weighted partition $( a_0, a_1, a_2, 0, 0, \cdots )$ of $q^3-q^2+q-2$, we see that $a_0\equiv q-2 = p-2+(p-1)p+\cdots+(p-1)p^{l-1}$ $\mod q$.
Let $a_i \equiv a_{i,0}+a_{i,1}p+\cdots+a_{i,l-1}p^{l-1}$ $\mod q$, where $i=1,2$ and $0\leq a_{i,j}\leq p-1$ for $j>1$.
It follows that if $C_{\underline{a}}\ne 0$, then $a_{1,0}+a_{2,0}\leq 1$ and $a_{1,j}+a_{2,j}=0$ for $j>0$. This implies
$\tilde{\underline{a}} = ( q-2,0,0,\cdots )$ or $( q-2,1,0,\cdots )$ or $( q-2,0,1,\cdots )$ and hence $S_{q^3-q^2+q-2}$ is the disjoint union of
$S_{q^3-q^2+q-2,(q-2,0,0)}$, $S_{q^3-q^2+q-2,(q-2,1,0)}$ and $S_{q^3-q^2+q-2,(q-2,0,1)}$. We have the following reductions, which are bijective:
\\
$$\begin{aligned}
&\underline{\tilde{\tilde{a}}}:S_{q^3-q^2+q-2,(q-2,0,0)}\rightarrow S_{q^2-q}\\
&\underline{\tilde{\tilde{a}}}:S_{q^3-q^2+q-2,(q-2,1,0)}\rightarrow S_{q^2-q-1}\\
&\underline{\tilde{\tilde{a}}}:S_{q^3-q^2+q-2,(q-2,0,1)}\rightarrow S_{q^2-2q}.
\end{aligned}$$
\\
Moveover, we see that $C_{\underline{a}}=C_{\underline{\tilde{\tilde{a}}}}$ if $\tilde{\underline{a}} = ( q-2,0,0,\cdots )$ and $C_{\underline{a}}=-C_{\underline{\tilde{\tilde{a}}}}$
if $\tilde{\underline{a}} = ( q-2,1,0,\cdots )$ or $( q-2,0,1,\cdots )$.
We obtain from Lemma \ref{inductionhn} that
$$\begin{aligned}
\frac{H_{q^3-q^2+q-2}}{\Gamma_{q^3-q^2+q-1}|_{\theta=t}} &= \sum_{\tilde{\underline{a}}}\sum_{\underline{a}\in S_{q^3-q^2+q-2,\tilde{\underline{a}}}}C_{\underline{a}}(\frac{G_0}{D_0|_{\theta=t}})^{a_0} \prod_{i\geq 1} (\frac{G_i}{D_i|_{\theta=t}})^{a_i}\\
&=\sum_{\tilde{\underline{a}}}\sum_{\underline{a}\in S_{q^3-q^2+q-2,\tilde{\underline{a}}}}C_{\underline{\tilde{a}}}C_{\underline{\tilde{\tilde{a}}}}(\frac{G_0}{D_0|_{\theta=t}})^{q\tilde{\tilde{a_0}}+\tilde{a_0}} \prod_{i\geq 1} (\frac{G_i}{D_i|_{\theta=t}})^{q\tilde{\tilde{a_i}}+\tilde{a_i}}\\
&=(\frac{H_{q^2-q}}{\Gamma_{q^2-q+1}|_{\theta=t}})^q-\frac{G_1}{D_1|_{\theta=t}}(\frac{H_{q^2-q-1}}{\Gamma_{q^2-q}|_{\theta=t}})^q-\frac{G_2}{D_2|_{\theta=t}}(\frac{H_{q^2-2q}}{\Gamma_{q^2-2q+1}|_{\theta=t}})^q
\end{aligned}$$
\\
\\
On the other hand,  by Lemma \ref{inductionhn} (b), we have
$$\frac{H_{q^2-q}}{\Gamma_{q^2-q+1}|_{\theta=t}}= \frac{H_{q^2-q-1}}{\Gamma_{q^2-q}|_{\theta=t}}+\frac{G_1}{D_1|_{\theta=t}}\frac{H_{q^2-2q}}{\Gamma_{q^2-2q+1}|_{\theta=t}}.$$
\\
It follows that

$$\begin{aligned}
\frac{H_{q^3-q^2+q-2}}{\Gamma_{q^3-q^2+q-1}|_{\theta=t}}
&=(\frac{H_{q^2-q}}{\Gamma_{q^2-q+1}|_{\theta=t}})^q-\frac{G_1}{D_1|_{\theta=t}}-\frac{G_2}{D_2|_{\theta=t}}
(\frac{D_1|_{\theta=t}}{G_1})^q(\frac{H_{q^2-q}}{\Gamma_{q^2-q+1}|_{\theta=t}}-1)^q\\
&=(1-\frac{G_2}{D_2|_{\theta=t}}(\frac{D_1|_{\theta=t}}{G_1})^q)(\frac{(t-\theta^q)^{q-1}}{L_1^{q-1}|_{\theta=t}})^q+
\frac{G_2}{D_2|_{\theta=t}}(\frac{D_1|_{\theta=t}}{G_1})^q-\frac{G_1}{D_1|_{\theta=t}}\\
&=(1-\frac{(t^{q^2}-\theta^q)}{[2]|_{\theta=t}})(\frac{(t-\theta^q)^{q-1}}{L_1^{q-1}|_{\theta=t}})^q+
\frac{(t^{q^2}-\theta^q)}{[2]|_{\theta=t}}-\frac{G_1}{D_1|_{\theta=t}}\\
&=\frac{-(t-\theta^q)^{q^2-q+1}}{[2][1]^{q(q-1)}|_{\theta=t}}+\frac{-(t-\theta^q)[1]^{q-1}|_{\theta=t}}{[2]|_{\theta=t}}
\end{aligned}$$
\\
By definition of $\Gamma$-function, we can easily derive that $$\Gamma_{q^3-q^2+q-1}= D_2^{q-1}$$ and the result follows.

\end{proof}
\subsection{{\small Formula for Anderson-Thakur polynomials $H_m$  with $m=q^n-\sum N_iq^i$}}
For $n\in \NN$, consider a tuple $(N_0,\cdots,N_{n-1})$ with $N_i \in \NN \cup \{0\}$ satisfying $0\leq \sum_{i=0}^{n-1}  N_i \leq q-1$. This implies
$q^n-\sum_{i=0}^{n-1} N_i q^i-1 \geq 0$. We have the following formula for these special polynomials:

\begin{theorem} \label{mainproposition} Let $n\in \NN$ and $N_i \in \NN \cup \{0\}$ satisfying $0\leq \sum_{i=0}^{n-1}  N_i \leq q-1$. Then
$$\displaystyle \frac{H_{q^n-\sum N_i q^i-1}}{\Gamma_{q^n-\sum N_i q^i}|_{\theta=t}} = \frac{(-1)^{\sum_{i=0}^{n-2} (n-1-i)N_iq^i}}{\prod_{i=0}^{n-2}L_{n-1-i}^{N_iq^i}}\prod_{i=1}^{n-1} (t-\theta^{q^i})^{\sum_{j=0}^{n-1-i} N_jq^j}\ \ \ \ \ \ \ \ \ \ \ \  (\ref{mainproposition}.1).$$
\end{theorem}

The key idea is using Lemma \ref{inductionhn} and Theorem \ref{Lucas} to descend Anderson-Thakur polynomials from $H_n$ to suitable $H_m$ with $m<n$.

\begin{proof}
Suppose that $q=p^l>2$.
We will prove by induction on $n$ and $N_i$. For $n=1$ it is clear that $\displaystyle\frac{H_{q-N_0-1}}{\Gamma_{q-N_0}|_{\theta=t}} = 1$ satisfying (\ref{mainproposition}.1) for any $0\leq N_0 \leq q-1$. Suppose that the statement holds for $H_{q^n-\sum N_i q^i-1}$ with $1\leq n<M$. Our goal is to
prove the formula (\ref{mainproposition}.1) holds for $H_{q^M-\sum N_i q^i-1}$. We split the proof into two steps.
\\
\\
$\underline{\mathbf{Step 1}}$ The formula holds for $H_{q^M-N-1}$ with $0\leq N \leq q-1$.
\\
\\
Now we start the case $H_{q^M-1}$ corresponding to $N_i=0$ for $i=0,
\cdots M-1$. Let $\underline{a} = (a_0,\cdots, a_{M-1},0, 0, \cdots )$ be a $q$ power weighted partition of $q^M-1$ with $C_{\underline{a}}\ne 0$, then
$a_0 \equiv q-1 = (p-1)+\cdots + (p-1)p^{l-1} \mod q$. By Lemma \ref{Lucas}, it forces $a_i\equiv 0 \mod q$.
Then by considering $\underline{\tilde{a}}= (q-1,0, 0, \cdots )$, we have that
$\underline{\tilde{\tilde{a}}}$ is a $q$ power weighted
partition of $q^{M-1}-1$. Conversely, if  $\underline{a}' \in S_{q^{M-1}-1}$, let
$\underline{a}=(qa_0'+q-1,qa_1',\cdots, qa_{M-1}', 0, 0, \cdots )$, then $\underline{a} \in S_{q^{M}-1}$ and $\underline{\tilde{\tilde{a}}}=\underline{a}'$. Therefore,
$$\underline{\tilde{\tilde{a}}}: S_{q^M-1}\rightarrow S_{q^{M-1}-1}$$ is bijective.  Moreover, $C_{\underline{a}} = C_{\underline{\tilde{\tilde{a}}}}$. It follows that

$$\begin{aligned}
\frac{H_{q^M-1}}{\Gamma_{q^M}|_{\theta=t}} &= \sum_{\underline{a}\in S_{q^M-1}}C_{\underline{a}}(\frac{G_0}{D_0|_{\theta=t}})^{a_0} \prod_{j\geq1} (\frac{G_i}{D_i|_{\theta=t}})^{a_i}\\
&= \sum_{\underline{a}'\in S_{q^{M-1}-1}}C_{\underline{a}'}(\frac{G_0}{D_0|_{\theta=t}})^{qa_0'+q-1} \prod_{j\geq1} (\frac{G_i}{D_i|_{\theta=t}})^{qa_i'}\\
&= (\frac{H_{q^{M-1}-1}}{\Gamma_{q^{M-1}}|_{\theta=t}})^q =1.
\end{aligned}$$
Here the last step is by induction hypothesis.
\\

Suppose that the formula (\ref{mainproposition}.1) holds for $q^M-N'-1$ with $0\leq N' <N \leq q-1$.
By Lemma \ref{inductionhn} (b) we have
$$\frac{H_{q^M-(N-1)-1}}{\Gamma_{q^M-(N-1)}|_{\theta=t}} =\frac{H_{q^M-N-1}}{\Gamma_{q^M-N-1}|_{\theta=t}} + \sum_{i=1}^{M-1} \frac{G_i}{D_i|_{\theta=t}}\frac{H_{q^M-(N-1)-q^i-1}}{\Gamma_{q^M-N-q^i+1}|_{\theta=t}}.$$

Note that $q^M-(N-1)-q^i-1 $ and $q^M-(N-1)-1$ are congruent to $q-N$ $\mod q$. If we consider all possible vectors $
\tilde{\underline{a}}= (q-N,\widetilde{a_1},\cdots, \tilde{a_{M-1}}, 0, 0, \cdots )$ with $0\leq \tilde{a_i}\leq q-1$ and $C_{\widetilde{\underline{a}}}\ne 0$, then
$\sum_{j=0}^{\infty} \tilde{a_j} = \sum_{j=0}^{M-1} \tilde{a_j} \leq q-1$, which implies $\sum_{j\geq 1}\tilde{a_j} \leq N-1\leq q-2$. It follows that
$$q^{M}-(N-1)-q^{i}-1-\sum_{j\geq 0}\tilde{a_j}q^{j} = q^M-q^i-\sum_{j\geq 1}\tilde{a_j}q^{j}-q\geq q^M-(q-1)q^{M-1}-q \geq 0.$$
\\
Similarly, $q^{M}-(N-1)-1-\displaystyle\sum_{j\geq 0}\tilde{a_j}q^{j}\geq 0.$ Hence for a given $\tilde{\underline{a}}$ we have two reduction maps
\\
$$\begin{aligned}
&\tilde{\tilde{\underline{a}}}: S_{q^M-(N-1)-q^i-1, \tilde{\underline{a}}} \rightarrow  S_{q^{M-1}-q^{i-1}-1-\sum_{j\geq 1}\widetilde{a_j}q^{j-1} }\\
&\tilde{\tilde{\underline{a}}}: S_{q^M-(N-1)-1, \tilde{\underline{a}}} \rightarrow  S_{q^{M-1}-1-\sum_{j\geq 1}\widetilde{a_j}q^{j-1} }.
\end{aligned}$$
\\
By argument similar to the above, we see further that they are bijective. Since $S_{q^M-(N-1)-q^i-1} = \displaystyle\bigcup_{\tilde{a}} S_{q^M-(N-1)-q^i-1, \tilde{\underline{a}}}$, by Lemma \ref{inductionhn} (a) we have
\\
$$\begin{aligned}
\frac{H_{q^M-(N-1)-q^i-1}}{\Gamma_{q^M-(N-1)-q^i}|_{\theta=t}} &= \sum_{\tilde{\underline{a}}}\sum_{\underline{a}\in S_{q^M-(N-1)-q^i-1,\tilde{\underline{a}}}}C_{\underline{a}}(\frac{G_0}{D_0|_{\theta=t}})^{a_0} \prod_{j\geq1} (\frac{G_i}{D_i|_{\theta=t}})^{a_i}\\
&=\sum_{\tilde{\underline{a}}}\sum_{\underline{a}\in S_{q^M-(N-1)-q^i-1,\tilde{\underline{a}}}}C_{\tilde{\underline{a}}}C_{\tilde{\tilde{\underline{a}}}}(\frac{G_0}{D_0|_{\theta=t}})^{q\tilde{\tilde{a_0}}+\tilde{a_0}} \prod_{j\geq1} (\frac{G_i}{D_i|_{\theta=t}})^{q\tilde{\tilde{a_i}}+\tilde{a_i}}\\
&=\sum_{\tilde{\underline{a}}}\sum_{\underline{a}'\in S_{q^{M-1}-q^{i-1}-1-\sum_{j\geq 1}\widetilde{a_j}q^{j-1}}}C_{\tilde{\underline{a}}}C_{\underline{a}'}(\frac{G_0}{D_0|_{\theta=t}})^{qa_0'+\tilde{a_0}} \prod_{j\geq1} (\frac{G_i}{D_i|_{\theta=t}})^{qa_i'+\tilde{a_i}}\\
&=\sum_{\tilde{\underline{a}}}C_{\tilde{\underline{a}}}\prod_{j\geq 0} (\frac{G_i}{D_i|_{\theta=t}})^{\tilde{a_i}}(\frac{H_{q^{M-1}-q^{i-1}-1-\sum_{j\geq 1}\tilde{a_j}q^{j-1}}}{\Gamma_{q^{M-1}-q^{i-1}-\sum_{j\geq 1}\tilde{a_j}q^{j-1}}|_{\theta=t}})^q
\end{aligned}$$
Similarly, $$\frac{H_{q^M-(N-1)-1}}{\Gamma_{q^M-(N-1)}|_{\theta=t}}= \sum_{\tilde{\underline{a}}}C_{\tilde{\underline{a}}}\prod_{j\geq 0} (\frac{G_i}{D_i|_{\theta=t}})^{\tilde{a_i}}(\frac{H_{q^{M-1}-1-\sum_{j\geq 1}\widetilde{a_j}q^{j-1}}}{\Gamma_{q^{M-1}-\sum_{j\geq 1}\widetilde{a_j}q^{j-1}}|_{\theta=t}})^q.$$
Now for each $\tilde{\underline{a}}$, $q^{M-1}-q^{i-1}-1-\sum_{j\geq 1}\tilde{a_j}q^{j-1}$ can be written as $q^{M-1}-\sum_{j=0}^{M-2} \tilde{N_j}q^j-1,$ where
$$\tilde{N_j}= \tilde{a_{j+1}}\ \mbox{if}\ j\ne i-1\ \mbox{and}\ \tilde{N_{i-1}}= \tilde{a_{i}}+ 1.$$ Since $\sum_{j=0}^{M-2} \tilde{N_j} = 1+\sum_{j\geq 1}\tilde{a_j}\leq q-N+\sum_{j\geq 1}\tilde{a_j}=\sum_{j\geq 0}\tilde{a_j}\leq q-1$. So by induction hypothesis on $M-1<M$, one can show that
$$\frac{H_{q^{M-1}-q^{i-1}-1-\sum_{j\geq 1}\widetilde{a_j}q^{j-1}}}{\Gamma_{q^{M-1}-q^{i-1}-\sum_{j\geq 1}\widetilde{a_j}q^{j-1}}|_{\theta=t}} = \frac{H_{q^{M-1}-1-\sum_{j\geq 1}\widetilde{a_j}q^{j-1}}}{\Gamma_{q^{M-1}-\sum_{j\geq 1}\widetilde{a_j}q^{j-1}}|_{\theta=t}} \frac{(-1)^{M-2-(i-1)q^{i-1}}}{L_{M-2-(i-1)}^{q^{i-1}}}\prod_{j=1}^{M-2-(i-1)}(t-\theta^{q^j})^{q^{i-1}}.$$
It follows that
$$\frac{H_{q^M-(N-1)-q^i-1}}{\Gamma_{q^M-(N-1)-q^i}|_{\theta=t}} = \frac{H_{q^M-(N-1)-1}}{\Gamma_{q^M-(N-1)}|_{\theta=t}} \frac{(-1)^{M-i-1}}{L_{M-i-1}^{q^{i}}}\prod_{j=1}^{M-i-1}(t-\theta^{q^j})^{q^{i}}.$$
It remains to prove that
$$1-\sum_{i=1}^{M-1} \frac{G_i}{D_i|_{\theta=t}}\frac{(-1)^{M-i-1}}{L_{M-i-1}^{q^{i}}}\prod_{j=1}^{M-i-1}(t-\theta^{q^j})^{q^{i}}=
\frac{(-1)^{M-1}}{L_{M-1}}\prod_{j=1}^{M-1}(t-\theta^{q^j}).$$
Let $$f(\theta):=1-\sum_{i=0}^{M-1} \frac{G_i}{D_i|_{\theta=t}}\frac{(-1)^{M-i-1}}{L_{M-i-1}^{q^{i}}}\prod_{j=1}^{M-i-1}(t-\theta^{q^j})^{q^{i}}$$
Let $U$ denotes the collection of  all subsets of $\{1,2,\cdots,M-1\}$. For $I=\{i_1,\cdots,i_m \}\in U$, we put $\theta^I = \theta^{q^{i_1}+\cdots+q^{i_m}}$,
and $|I|=m$, the number of elements in $I$. Then for $i=0,\cdots,M-1$, $$G_i\prod_{j=1}^{M-i-1}(t-\theta^{q^j})^{q^{i}}=\prod_{j=1}^{M-1} (t^{q^i}-\theta^{q^j}) = \sum_{I\in U} \theta^I (-1)^{|I|}(t^{q^i})^{M-1-|I|}.$$ Since for distinct $I_1, I_2$, $\theta^{I_1}\ne \theta^{I_2}$, we have $$f(\theta) =1-\sum_{I\in U}(\sum_{i=0}^{M-1} \frac{(t^{q^i})^{M-1-|I|}}{(-1)^{M-i-1}D_iL_{M-i-1}^{q^i}|_{\theta=t}}) \theta^I (-1)^{|I|}.$$
Observe that $$\frac{1}{(-1)^{M-i-1}L_{M-i-1}^{q^i}|_{\theta=t}}= \frac{\prod_{j=1}^{i}(t^{q^i}-t^{q^{j+M-1}})}{{(-1)^{M-1}}L_{M-1}^{q^i}|_{\theta=t}} .$$
It follows that $$f(\theta) =1-\sum_{I\in U}(\sum_{i=0}^{M-1} \frac{\prod_{j=1}^{i}(t^{q^i}-t^{q^{j+M-1}})(t^{q^i})^{M-1-|I|}}{(-1)^{M-1}D_iL_{M-1}^{q^i}|_{\theta=t}}) \theta^I (-1)^{|I|}. = 1- \sum_{I\in U} \Psi_{M-1}|_{\theta=t}(t^{M-1-|I|})\theta^I (-1)^{|I|}. $$
By Proposition \ref{ATbinomial}, $\Psi_{M-1}|_{\theta=t}(t^{M-1-|I|})= 0$ if $M-1-|I| <M-1$, that is to say $|I|>0$, and $\Psi_{M-1}|_{\theta=t}(t^{M-1-|I|})= 1$ if  $|I|=0$, that is to say $I$ is the empty set. In the later condition $\theta^I (-1)^{|I|}=1$ and hence $f(\theta) =1-1 =0$.
By the claim we deduce that
\\
$$\begin{aligned}
\frac{H_{q^M-N-1}}{\Gamma_{q^M-N}|_{\theta=t}} &= \frac{H_{q^M-(N-1)-1}}{\Gamma_{q^M-(N-1)}|_{\theta=t}}\frac{(-1)^{M-1}}{L_{M-1}}\prod_{j=1}^{M-1}(t-\theta^{q^j})\\
&=\frac{(-1)^{ (M-1)(N-1)}}{L_{M-1}^{N-1}}\prod_{i=1}^{M-1} (t-\theta^{q^i})^{N-1} \frac{(-1)^{M-1}}{L_{M-1}}\prod_{j=1}^{M-1}(t-\theta^{q^j})\\
&=\frac{(-1)^{ (M-1)N}}{L_{M-1}^{N}}\prod_{i=1}^{M-1} (t-\theta^{q^i})^{N}.
\end{aligned}$$
\\
\\
$\underline{\mathbf{Step 2}}$ The formula holds for $H_{q^M-\sum_{i=0}^{D} N_iq^i-1}$ with $1\leq D\leq M-1$.
\\
\\
Suppose that $H_{q^M-\sum_{i=0}^{d} N_i'q^i-1}$ satisfies the formula (\ref{mainproposition}.1) for $0\leq d<D\leq M-1$ and $0\leq \sum_{i=0}^{d} N_i'q^i \leq q-1$. We will prove that $H_{q^M-\sum_{i=0}^{D} N_iq^i-1}$ satisfies the formula (\ref{mainproposition}.1) for $0\leq \sum_{i=0}^{D} N_iq^i \leq q-1$ with $N_D \ne 0$. For $1\leq N_D \leq q-1$ and $\sum_{i=0}^{D} N_iq^i \leq q-1$, we have
$$q^M-\sum_{i=0}^{D-1} N_iq^i -(N_D-1)q^D -1 \equiv q^M-\sum_{i=0}^{D-1} N_iq^i -N_Dq^D -1 \equiv q-(N_0 +1) \mod q.$$
\\
Put $N_k = 0$  for $k>D$. Then for any tuple $
\tilde{\underline{a}}= (q-(N_0+1),\tilde{a_1},\cdots, \tilde{a_{M-1}})$ with $C_{\tilde{\underline{a}}}\ne 0$,
$\sum_{j=1}^{M-1} \tilde{a_j} \leq  q-1- \tilde{a_0}= N_0$. Also we have $\sum_{j\geq 1}N_j \leq q-1-N_0 $. This implies that
$$q^{M}-\sum_{j\geq 0}(N_j+\tilde{a_j})q^{j}-1 =q^{M}-\sum_{j\geq 1}(N_j+\tilde{a_j})q^{j}-q \geq q^{M}-(q-1)q^{M-1}-q \geq 0 $$ for $M\geq 2$.
Similarly we have $q^{M}+q^{D}-\sum_{j\geq 0}(N_j+\tilde{a_j})q^{j}-1\geq 0 $.  It follows that
\\
$$\begin{aligned}
&\tilde{\tilde{\underline{a}}}: S_{q^M-\sum_{i=0}^{D-1} N_iq^i -(N_D-1)q^D-1, \tilde{\underline{a}}} \rightarrow  S_{q^{M-1}+q^{D-1}-\sum_{j\geq 1}(N_j+\tilde{a_j})q^{j-1} -1 }\\
&\tilde{\tilde{\underline{a}}}: S_{q^M-\sum_{i=0}^{D-1} N_iq^i -N_Dq^D-1, \tilde{\underline{a}}} \rightarrow  S_{q^{M-1}-\sum_{j\geq 1}(N_j+\tilde{a_j})q^{j-1} -1}
\end{aligned}$$
are bijective. By Lemma \ref{inductionhn}(a) we have
\\
$$\begin{aligned}
\frac{H_{q^M-\sum_{i=0}^{D-1} N_iq^i -(N_D-1)q^D-1}}{\Gamma_{q^M-\sum_{i=0}^{D-1} N_iq^i -(N_D-1)q^D}|_{\theta=t}} &= \sum_{\tilde{\underline{a}}}\sum_{\underline{a}\in S_{q^M-\sum_{i=0}^{D-1} N_iq^i -(N_D-1)q^D-1,\tilde{\underline{a}}}}(\frac{G_0}{D_0|_{\theta=t}})^{a_0} \prod_{j\geq1} (\frac{G_i}{D_i|_{\theta=t}})^{a_i}\\
&=\sum_{\tilde{\underline{a}}}C_{\tilde{\underline{a}}}\prod_{j\geq 0} (\frac{G_i}{D_i|_{\theta=t}})^{\tilde{a_i}}(\frac{H_{q^{M-1}+q^{D-1}-\sum_{j\geq 1}(N_j+\tilde{a_j})q^{j-1} -1}}{\Gamma_{q^{M-1}+ q^{D-1}-\sum_{j\geq 1}(N_j+\tilde{a_j})q^{j-1}}|_{\theta=t}})^q
\end{aligned}$$
\\
and
\\
$$\begin{aligned}
\frac{H_{q^M-\sum_{i=0}^{D-1} N_iq^i -N_Dq^D-1}}{\Gamma_{q^M-\sum_{i=0}^{D-1} N_iq^i -N_Dq^D}|_{\theta=t}} &= \sum_{\tilde{\underline{a}}}\sum_{\underline{a}\in S_{q^M-\sum_{i=0}^{D-1} N_iq^i -(N_D-1)q^D-1,\tilde{\underline{a}}}}(\frac{G_0}{D_0|_{\theta=t}})^{a_0} \prod_{j\geq1} (\frac{G_i}{D_i|_{\theta=t}})^{a_i}\\
&=\sum_{\tilde{\underline{a}}}C_{\tilde{\underline{a}}}\prod_{j\geq 0} (\frac{G_i}{D_i|_{\theta=t}})^{\tilde{a_i}}(\frac{H_{q^{M-1}-\sum_{j\geq 1}(N_j+\tilde{a_j})q^{j-1} -1}}{\Gamma_{q^{M-1}-\sum_{j\geq 1}(N_j+\tilde{a_j})q^{j-1}}|_{\theta=t}})^q
\end{aligned}$$
\\
\\
By induction hypothesis on $M-1<M$, we deduce that
$$\frac{H_{q^M-\sum_{i=0}^{D-1} N_iq^i -N_Dq^D-1}}{\Gamma_{q^M-\sum_{i=0}^{D-1} N_iq^i -N_Dq^D}|_{\theta=t}}=
\frac{H_{q^M-\sum_{i=0}^{D-1} N_iq^i -(N_D-1)q^D-1}}{\Gamma_{q^M-\sum_{i=0}^{D-1} N_iq^i -(N_D-1)q^D}|_{\theta=t}} \frac{(-1)^{(M-D-1)q^{D}}}{L_{M-D-1}^{q^{D}}}[(t-\theta^q)\cdots (t-\theta^{q^{M-D-1}})]^{q^{D}}.$$ So if we start from the case $N_D=1$,
we can prove that the formula (\ref{mainproposition}.1) holds for $1\leq N_D \leq q-1$ by the above relation and the induction hypothesis of
$q^M-\sum_{i=0}^{D-1} N_i'q^i-1$.

\end{proof}

\section{Main result on Zeta-like Multizeta values}\label{zeta-like}
In this section we will prove Conjecture \ref{LRTconj} (b) with $q>2$ and Conjecture \ref{LKconj}.
\subsection{{\small Frobenius twisting}}
We fix the following automorphism of the field of Laurent series over $\CC_{\infty}$, which is referred to as \emph{Frobenius twisting}:
\[
     \begin{array}{rcl}
      \laurent{\CC_\infty}{t}  & \rightarrow & \laurent{\CC_\infty}{t},\\
       f:=\sum_{i}a_{i}t^{i} & \mapsto & f^{(-1)}:=\sum_{i}{a_{i}}^{\frac{1}{q}}t^{i}. \\
     \end{array}
 \]
\\
In \cite{CPY}, the following criterion is proved for deciding zeta-like multizeta values in terms of Anderson-Thakur polynomials.
\begin{theorem} $($ Chang-Papanikolas-Yu \cite[Theorem 2.5.2, 4.4.2]{CPY} $)$ \label{CPY}
Given a tuple $(s_1,s_2,\cdots, s_r)$, then
$\zeta(s_1,\cdots, s_r)$ is zeta-like if and only if there exist  $\delta_{1}, \cdots , \delta_{r}\in \overline{K}[t]$ and $a,b\in \FF_q[t]$ such that
\begin{equation}\label{E:Eqdelta}
\begin{split}
  \delta_{1} &= {\textstyle \delta_{1}^{(-1)}(t-\theta)^{w}+\delta_{2}^{(-1)}H_{s_1-1}^{(-1)}(t-\theta)^{w}+bH_{w-1}^{(-1)}(t-\theta)^{w}   ;} \\
   \delta_{2} &= {\textstyle \delta_{2}^{(-1)}(t-\theta)^{s_{2}+\cdots+s_{r}}+\delta_{3}^{(-1)}H_{s_2-1}^{(-1)}(t-\theta)^{s_{2}+\cdots+s_{r}}  ;} \\
    &\quad\quad\quad\quad\quad\quad\quad\quad\quad\vdots \\
\delta_{r-1} &= {\textstyle \delta_{r-1}^{(-1)}(t-\theta)^{s_{r-1}+s_{r}}+\delta_{r}^{(-1)}H_{s_{r-1}-1}^{(-1)}(t-\theta)^{s_{r-1}+s_{r}}  ;}\\
\delta_{r} &={\textstyle \delta_{r}^{(-1)}(t-\theta)^{s_{r}}+aH_{s_r-1}^{(-1)}(t-\theta)^{s_{r}}   ,}
\end{split}
\end{equation}
where $H_{s_1-1},\cdots, H_{s_r-1}, H_{w-1}$ are Anderson-Thakur polynomials.
\\
Furthermore, if $q-1$ does not divide $\sum s_i$, then we have $$a(\theta)\Gamma_{s_1}\cdots\Gamma_{s_r}\zeta_A(s_1,\cdots, s_r)+b(\theta)\Gamma_{w}\zeta_A(w)=0$$
\end{theorem}

\begin{remark}
If $( \delta_{1}, \cdots , \delta_{r}, a, b )$ are solutions of (\ref{E:Eqdelta}), then for any nonzero $f\in \FF_q[t]$, $( f\delta_{1}, \cdots , f\delta_{r}, fa, fb )$ is also a solution of $(\ref{E:Eqdelta})$.
\end{remark}
Basing on this theorem, our strategy for proving given multizeta values to be zeta-like is to actually solve system of Equations \ref{E:Eqdelta} by finding $\delta_{1}, \cdots , \delta_{r}\in \overline{K}[t]$ and $a,b\in \FF_q[t]$.
Since we are interested in tuples $(s_1,\cdots, s_r)$ with $s_i$ of very special $q$-adic \lq\lq shape \rq\rq, solution $(a, \delta_r)$  can be given immediately. Then an inductive procedure is used to go from a solution $(a', \delta_j', \cdots, \delta_r')$ of a subsystem of (\ref{E:Eqdelta}) with $r-j+1$ equations to a solution $(a, \delta_{j-1}, \cdots, \delta_r)$ of a subsystem of (\ref{E:Eqdelta}) with $r-j+2$ equations. This is a content of the following proposition.
\begin{proposition}\label{mainproposition2}
Fix $1\leq p^M\leq q$. For any $r\geq 2$ and $m\in \NN\cup \{0\}$, let $s_i = p^M(q-1)q^{m+i-2}$ for $i=2,\cdots r$.
Let $(a, \delta_2, \cdots, \delta_r)$ be defined as follows:
\begin{equation}\label{S:Solutionsub}
\begin{split}
    f_r&:={\textstyle [2]^{p^Mq^{r+m-3}}\cdots [r-1]^{p^Mq^m} \Gamma_{p^Mq^{m+r-2}(q-1)} ;}\\
    f_i&:={\textstyle  \displaystyle\frac{- f_{i+1}}{[r-i+1]^{p^Mq^{m+i-2}}} \Gamma_{p^Mq^{m+i-2}(q-1)} \ \mbox{for}\ j\leq i<r;}\\
   \delta_{i} &:= {\textstyle f_i|_{\theta =t} [(t-\theta)\cdots(t-\theta^{\frac{1}{q^{r-i}}})]^{p^Mq^{r+m-1}} ;} \\
 a&:={\textstyle -\left[[1]^{p^Mq^{r+m-2}}[2]^{p^Mq^{r+m-3}}\cdots [r-1]^{p^Mq^m}\right]|_{\theta =t} .}
\end{split}
\end{equation}
\\
Then for any $j$ with $2\leq j< r$, the system of equations
\begin{equation}\label{E:Eqdeltasub}
\begin{split}
   \delta_{j} &= {\textstyle \delta_{j}^{(-1)}(t-\theta)^{s_{j}+\cdots+s_{r}}+\delta_{j+1}^{(-1)}H_{s_j-1}^{(-1)}(t-\theta)^{s_{j}+\cdots+s_{r}}  ;} \\
    &\quad\quad\quad\quad\quad\quad\quad\quad\quad\vdots \\
\delta_{r-1} &= {\textstyle \delta_{r-1}^{(-1)}(t-\theta)^{s_{r-1}+s_{r}}+\delta_{r}^{(-1)}H_{s_{r-1}-1}^{(-1)}(t-\theta)^{s_{r-1}+s_{r}}  ;}\\
\delta_{r} &={\textstyle \delta_{r}^{(-1)}(t-\theta)^{s_{r}}+aH_{s_r-1}^{(-1)}(t-\theta)^{s_{r}}   .}
\end{split}
\end{equation}
\\
can be solved explicitly with $(a, \delta_j, \cdots, \delta_r)$ given by $(\ref{S:Solutionsub})$.
\end{proposition}
\begin{remark}\label{f2}\
It follows from the recursive definition of $f_i$ that
$$\displaystyle f_2 =(-1)^r\Gamma_{p^M(q-1)q^m}\cdots \Gamma_{p^M(q-1)q^{m+r-2}}.$$
\end{remark}
\begin{proof}
By Theorem \ref{mainproposition}, we see that for $2\leq i \leq r$ and $p^M=1$ or $q$,
$$H_{s_i-1} = \Gamma_{s_i}|_{\theta =t} = \Gamma_{p^Mq^{m+i-2}(q-1)}.$$
\\
For $1<p^M < q$, since $p^Mq^{I}(q-1)-1 \equiv q-1$ if $I\geq 1$, by taking $\tilde{\underline{a}}= (q-1, 0, 0,\cdots )$ we have the following reduction:
\\
\\
$$\begin{aligned}
&\tilde{\tilde{\underline{a}}}: S_{p^Mq^{m+i-2}(q-1)-1} \rightarrow S_{p^Mq^{m+i-3}(q-1)-1}\\
&\tilde{\tilde{\underline{a}}}^{2}: S_{p^Mq^{m+i-3}(q-1)-1} \rightarrow S_{p^Mq^{m+i-4}(q-1)-1}\\
&\hspace{120pt} \vdots \\
&\tilde{\tilde{\underline{a}}}^{m+i-2}: S_{p^Mq(q-1)-1} \rightarrow S_{p^M(q-1)-1}.
\end{aligned}$$
\\
Here $\tilde{\tilde{\underline{a}}}^{n}$ means the iteration of $\tilde{\tilde{\underline{a}}}$ by $n$ times. Using this sequence of reduction maps we
have
$$\frac{H_{p^Mq^{m+i-2}(q-1)-1}}{\Gamma_{p^Mq^{m+i-2}(q-1)}|_{\theta =t}} =  (\frac{H_{p^M(q-1)-1}}{\Gamma_{p^M(q-1)}|_{\theta =t}})^{q^{m+i-2}}.$$
\\
Write $p^Mq-p^M-1 = a_0+a_1q$. Then $a_0 \equiv q-1-p^M \mod q$ and hence $a_0 \equiv 0, p^M \mod q$ with  $0\leq a_1 <p^M$. This implies $C_{\underline{a}}\ne 0$ if and only if
$(a_0 ,a_1, \cdots) = (p^Mq-p^M-1, 0, 0, \cdots)$. Hence $\displaystyle\frac{H_{p^M(q-1)-1}}{\Gamma_{p^M(q-1)}|_{\theta =t}}$ equals $1$ and so does $\displaystyle\frac{H_{p^Mq^{m+i-2}(q-1)-1}}{\Gamma_{p^Mq^{m+i-2}(q-1)}|_{\theta =t}}$.
\\
\\
$(i)$ $\delta_{r} =\delta_{r}^{(-1)}(t-\theta)^{s_{r}}+aH_{s_r-1}^{(-1)}(t-\theta)^{s_{r}} $
\\
$$\begin{aligned}
\delta_{r}^{(-1)}(t-\theta)^{s_{r}}+aH_{s_r-1}^{(-1)}(t-\theta)^{s_{r}}&=
[f_r|_{\theta =t}(t-\theta^{\frac{1}{q}})^{p^Mq^{r+m-1}}+a\Gamma_{p^Mq^{m+r-2}(q-1)}|_{\theta =t}](t-\theta)^{p^M(q-1)q^{m+r-2}}\\
&=[f_r|_{\theta =t}(t-\theta^{\frac{1}{q}})^{p^Mq^{r+m-1}}-f_r[1]^{p^Mq^{r+m-2}}](t-\theta)^{p^M(q-1)q^{m+r-2}}\\
&=f_r(t-\theta)^{p^Mq^{m+r-2}}(t-\theta)^{p^M(q-1)q^{m+r-2}} = \delta_r.
\end{aligned}$$
\\
\\
$(ii)$ $\delta_{i} =\delta_{i}^{(-1)}(t-\theta)^{s_i+\cdots + s_{r}}+\delta_{i+1}H_{s_i-1}^{(-1)}(t-\theta)^{s_i+\cdots +s_{r}} $
\\
$$\begin{aligned}
&\ \delta_{i}^{(-1)}+ \delta_{i+1}^{(-1)}H_{s_{i}-1}^{(-1)}\\
&=f_i|_{\theta =t}[(t-\theta^{\frac{1}{q}})\cdots(t-\theta^{\frac{1}{q^{r+1-i}}})]^{p^Mq^{m+r-1}} +f_{i+1}|_{\theta =t}[(t-\theta^{\frac{1}{q}})\cdots(t-\theta^{\frac{1}{q^{r-i}}})]^{p^Mq^{m+r-1}} H_{p^M(q-1)q^{m+i-2}-1}^{(-1)}\\
&=[(t-\theta^{\frac{1}{q}})\cdots(t-\theta^{\frac{1}{q^{r-i}}})]^{p^Mq^{m+r-1}}[ f_i|_{\theta =t}(t-\theta^{\frac{1}{q^{r+1-i}}})^{p^Mq^{m+r-1}}+ f_{i+1}|_{\theta =t} \Gamma_{p^Mq^{m+i-2}(q-1)}|_{\theta =t}]\\
&=[(t-\theta^{\frac{1}{q}})\cdots(t-\theta^{\frac{1}{q^{r-i}}})]^{p^Mq^{m+r-1}}[ f_i|_{\theta =t}(t-\theta^{\frac{1}{q^{r+1-i}}})^{p^Mq^{m+r-1}}- f_{i}|_{\theta =t} [r-i+1]^{p^Mq^{m+i-2}}|_{\theta =t}]\\
&=[(t-\theta^{\frac{1}{q}})\cdots(t-\theta^{\frac{1}{q^{r-i}}})]^{p^Mq^{m+r-1}}f_{i}|_{\theta =t} [(t^{p^Mq^{m+r-1}}-\theta^{p^Mq^{m+i-2}})-(t^{p^Mq^{m+r-1}}-t^{p^Mq^{m+i-2}})]\\
&=[(t-\theta^{\frac{1}{q}})\cdots(t-\theta^{\frac{1}{q^{r-i}}})]^{p^Mq^{m+r-1}}f_{i}|_{\theta =t}(t-\theta)^{{p^Mq^{m+i-2}}}.
\end{aligned}$$
\\
\\
Hence
\\
$$\begin{aligned}
[\delta_{i}^{(-1)}+ \delta_{i+1}^{(-1)}H_{s_{i}-1}^{(-1)}](t-\theta)^{s_i+\cdots+s_r} &= [(t-\theta^{\frac{1}{q}})\cdots(t-\theta^{\frac{1}{q^{r-i}}})]^{p^Mq^{m+r-1}}f_{i}|_{\theta =t}(t-\theta)^{p^Mq^{m+i-2}}(t-\theta)^{p^M(q^{r-i+1}-1)q^{m+i-2}}\\
&= \delta_i
\end{aligned}$$
\\
\end{proof}
\

Now we begin to prove Conjecture \ref{LRTconj} (b) and Conjecture \ref{LKconj} (a).
\begin{theorem}\label{mainthm3} Suppose that $q>2$.
\\
\\
$(a)$ $\zeta_A(1,q^2-1,(q-1)q^2,\cdots,(q-1)q^{n+1})$ is zeta-like. In particular, we  have
$$\zeta_A(1,q^2-1,(q-1)q^2,\cdots,(q-1)q^{n+1}) = \displaystyle\frac{(-1)^{n+1}([n+2]-1)}{[1][n+2]}\frac{1}{[1]^{q^{n+1}}\cdots [n]^{q^2}}\zeta_A(q^{n+2}).$$
\\
$(b)$ For $n\geq 1$,  $r\geq 2$ and $1\leq p^m\leq q$, consider $N_i\in \NN \cup \{0\}$ for $0\leq i \leq n-1$ such that  $1\leq \sum N_i \leq q-1$. If
$(q-1)(q^n-\sum N_iq^i)\leq p^m(q-1)q^{n-1}$, then
$\zeta_A(q^n-\sum N_iq^i,p^m(q-1)q^{n-1},\cdots, p^m(q-1)q^{n+r-3})$ is zeta-like. In particular, if $q-1$ does not divide $q^n-\sum N_iq^i$, then we  have
\\
$$\begin{aligned}
&\ \zeta_A(q^n-\sum N_iq^i, p^m(q-1)q^{n-1},\cdots,p^m(q-1)q^{n+r-3})\\
&= \displaystyle\frac{(-1)^{(r-1)(1+\sum N_i)}L_{n+r-2}^{N_0q^0}\cdots L_{r-1}^{(N_{n-1}-(q-p^m))q^{n-1}}}{[1]^{p^mq^{n+r-3}} [2]^{p^mq^{n+r-4}} \cdots [r-1]^{p^mq^{n-1}}L_{n-1}^{N_0q^0}\cdots L_1^{N_{n-2}q^{n-2}}}\zeta_A(p^mq^{n+r-2}-p^mq^{n-1}+q^n-\sum N_iq^i).
\end{aligned}$$

\end{theorem}
\
\\
\begin{remark}\label{coverLT}
The zeta-like part of Theorem $3.1$ $(1)$ in \cite{LRT} is a special case of Theorem \ref{mainthm3} $(b)$ by taking $r=2$. Also, the zeta-like part of Theorem $3.2$ in \cite{LRT} is a special case of Theorem \ref{mainthm3} $(b)$  by taking $n=0, N_i=0$.
\end{remark}

\begin{proof}
\
\\
\\
$(a)$ Let $w=q^{n+2}$. Consider $p^M=1$, $m=1$ and $r=n+2$ in Proposition \ref{mainproposition2}. Let $s_1=1$, $s_2=q^2-1$ and $s_i =(q-1)q^{i-1}$ for $3\leq i\leq n+2$.
If $n+2\geq 3$, we choose $f_i\in A$, $\delta_i\in \overline{K}[t]$ and $a\in \FF_q[t]$ the same as in Proposition \ref{mainproposition2} when $3\leq i\leq n+2$.
Then by Proposition \ref{mainproposition2}, $\delta_i$ and $a$ satisfy subsystem of equations (\ref{S:Solutionsub}) for $j=3$.
If $n+2=2$, we define $\delta_3 = a = -[1]^{q}|_{\theta =t}$.
Now we let
\\
\\
$\begin{aligned}
\delta_2&=(-1)^n [n+1]^q|_{\theta =t} \Gamma_{w}|_{\theta =t} [(t-\theta)\cdots(t-\theta^{\frac{1}{q^{n}}})]^{q^{n+2}}(t-\theta^q) \in A[t].\\
f&=\{-[n+2]L_1L_2\cdots L_{n+1}\}|_{\theta =t}\\
b&=(-1)^n([n+2]-1)|_{\theta =t}[n+1]|_{\theta =t}^q.
\end{aligned}$
\
\\
\\
We further put
\\
\\
$\begin{aligned}
&B=(-1)^{n+1}[n+1]^q|_{\theta =t}\Gamma_{w}|_{\theta =t}\\
&F_0=B(t-\theta)^{w},\\
&F_1=B(t-\theta)^{w}(t-\theta^{\frac{w}{q}}),\\
&F_{i}=F_{i-1}^{(-1)}(t-\theta)^w \ \ \ \ \mbox{for}\ i=2,\cdots n+1.
\end{aligned}$
\\
and let $\delta_1 = \displaystyle\sum_{j=0}^{n+1}F_j$. From this recursive formula we can see that $\delta_2 = -F_{n+1}$
Our goal is to prove that  $(fa,b,\delta_1,\delta_2,f\delta_3,\cdots,f\delta_{n+2})$ satisfies the system of equations (\ref{E:Eqdelta}).
\\
\\
Part I:
$\delta_2 = \delta_{2}^{(-1)}(t-\theta)^{s_{2}+\cdots+s_{n+2}}+ f\delta_{3}^{(-1)}H_{s_{2}-1}^{(-1)}(t-\theta)^{s_{2}+\cdots+s_{n+2}}$
\
\\
\\
Note that by Theorem \ref{mainproposition}, $$H_{s_{2}-1} = H_{q^{2}-2} = \Gamma_{q^2-1}\frac{-(t-\theta^q)}{L_1}= D_1^{q-2}(\theta-t).$$ Hence we have
\\
\\
$\begin{aligned}
&\ \delta_{2}^{(-1)}(t-\theta)^{s_{2}+\cdots+s_{n+2}}+ f\delta_{3}^{(-1)}H_{s_{2}-1}^{(-1)}(t-\theta)^{s_{2}+\cdots+s_{n+2}}\\
&=(t-\theta)^{q^{n+2}-1}(t-\theta^{\frac{1}{q}})^{q^{n+2}}\cdots(t-\theta^{\frac{1}{q^{n}}})^{q^{n+2}}\left[-B(t-\theta)(t-\theta^{\frac{1}{q^{n+1}}})^{q^{n+2}} +ff_3|_{\theta =t}(\theta-t)D_1^{q-2}|_{\theta =t}\right]\\
\end{aligned}$
\
\\
\\
Since $ff_3|_{\theta =t}=(-1)^{n}[n+2][n+1]^q|_{\theta =t}(D_{n+1}D_{n}\cdots D_{2})^{q-1}L_1=\displaystyle\frac{-B[n+2]|_{\theta=t}}{D_1^{q-2}|_{\theta=t}},$ we have
\
\\
\\
$\begin{aligned}
&\ \delta_{2}^{(-1)}(t-\theta)^{s_{2}+\cdots+s_{n+2}}+ f\delta_{3}^{(-1)}H_{s_{2}-1}^{(-1)}(t-\theta)^{s_{2}+\cdots+s_{n+2}}\\
&=-B (t-\theta)^{q^{n+2}}[(t-\theta^{\frac{1}{q}})\cdots(t-\theta^{\frac{1}{q^{n}}})]^{q^{n+2}}\left[(t-\theta^{\frac{1}{q^{n+1}}})^{q^{n+2}}- [n+2]\right]\\
&=\delta_{2}.
\end{aligned}$
\
\\
\\
Part II : $\delta_{1} = \delta_{1}^{(-1)}(t-\theta)^{w}+\delta_{2}^{(-1)}H_{s_1-1}^{(-1)}(t-\theta)^{w}+bH_{w-1}^{(-1)}(t-\theta)^{w}$. \
\\
\\
$\begin{aligned}
&\delta_{1}^{(-1)}(t-\theta)^{w}+\delta_{2}^{(-1)}H_{s_1-1}^{(-1)}(t-\theta)^{w}+bH_{w-1}^{(-1)}(t-\theta)^{w}\\
&=\displaystyle\sum_{j=0}^{n+1}F_j^{(-1)}(t-\theta)^{w}-F_{n+1}^{(-1)}(t-\theta)^{w}+(-1)^n([n+2]-1)|_{\theta =t}[n+1]|_{\theta =t}^q\Gamma_{w}|_{\theta=t}(t-\theta)^{w}\\
&=F_0^{(-1)}(t-\theta)^{w}\displaystyle\sum_{j=2}^{n+1}F_{j}-B([n+2]-1)|_{\theta =t}(t-\theta)^{w}\\
&=B(t-\theta)^{w}(t-\theta^{\frac{1}{q}})^{w}+\delta_1-F_0-F_1-B[n+2]|_{\theta =t}(t-\theta)^{w}+F_0\\
&=\delta_1+B(t-\theta)^{w}\left[(t-\theta^{\frac{1}{q}})^{w}-(t-\theta^{\frac{w}{q}})-[n+2]|_{\theta =t}\right]=\delta_1.
\end{aligned}$
\
\\

Finally,  we compute the following $\Gamma$-functions.
\\
\\
$\begin{aligned}
&\Gamma_{w}=(D_1D_2\cdots D_{n+1})^{q-1}\\
&\Gamma_{q^2-1}=D_1^{q-1}\\
&\Gamma_{(q-1)q^{i}}=\frac{D_i^{q-1}}{L_i}\ \mbox{for}\ i=2,\cdots,n+1.
\end{aligned}$
\\
Since $q-1$ does not divide $\sum s_i$, we have
\\
$$\begin{aligned}
\zeta_A(1,q^2-1,q^2(q-1),\cdots, q^{n+1}(q-1)) &= \frac{-b(\theta)\Gamma_{w}}{f(\theta)a(\theta)\Gamma_{1}\Gamma_{q^{2}-1}\Gamma_{(q-1)q^{2}}\cdots\Gamma_{(q-1)q^{n+1}}}\zeta_A(w)\\
&=\frac{(-1)^{n+1}([n+2]-1)[n+1]^qL_2\cdots L_{n+1}}{[n+2]L_1L_2\cdots L_{n+1}[1]^{q^{n}}[2]^{q^{n-1}}\cdots [n+1]^{q}}\\
&=\frac{(-1)^{n+1}([n+2]-1)}{[n+2]L_1[1]^{q^{n}}[2]^{q^{n-1}}\cdots [n]^{q^2}}
\end{aligned}$$
\\
\\
(b) Let $w=s_1+\cdots+s_r=p^mq^{n+r-2}-p^mq^{n-1}+q^n-\sum N_iq^i$. For $b\in \FF_q[t]$, we put
\\
\\
$\begin{aligned}
F_0&=bH_{w-1}^{(-1)}(t-\theta)^w\\
F_i&=F_{i-1}^{(-1)}(t-\theta)^w \ \ \ \ \mbox{for}\ i=1,\cdots r-1\\
\delta_1&=\sum_{j=0}^{r-2} F_j \in \overline{K}[t].\\.
\end{aligned}$
\\
Then
\\
\\
$\begin{aligned}
&\delta_{1}^{(-1)}(t-\theta)^{w}+f_1\delta_{2}^{(-1)}H_{s_1-1}^{(-1)}(t-\theta)^{w}+bH_{w-1}^{(-1)}(t-\theta)^{w}\\
&=\sum_{j=0}^{r-2} F_j^{(-1)}(t-\theta)^{w}+\delta_{2}^{(-1)}H_{s_1-1}^{(-1)}(t-\theta)^{w}+F_0
\\
&=\delta_1+F_{r-1}+ f_1\delta_{2}^{(-1)}H_{s_1-1}^{(-1)}(t-\theta)^{w}.\\
\end{aligned}$
\\
\\
It suffices to show that there exists nonzero $b,f_1\in \FF_q[t]$ such that
$$F_{r-1}+ f_1\delta_{2}^{(-1)}H_{s_1-1}^{(-1)}(t-\theta)^{w}= 0,$$ or equivalently,
$$f_1\delta_{2}^{(-1)}H_{s_1-1}^{(-1)}(t-\theta)^{w}=-bH_{w-1}^{(-r)}(t-\theta)^w\cdots(t-\theta^{\frac{1}{q^{r-1}}})^w .$$
\\
Here $f_1$ plays the role adjusting the solution $(\delta_2,\cdots,\delta_r,a)$ in (\ref{S:Solutionsub}).
Note that  the condition $(q-1)s_1\leq s_2$ is equivalent to
$$q^{n-1}(q-p^m) \leq \sum_{i=0}^{n-1} N_iq^i,$$ which implies $N_{n-1}\geq q-p^m$. Therefore we can rewrite the weight as
$$w= q^{n+r-1}- (q-p^m)q^{n+r-2} -(N_{n-1}-(q-p^m))q^{n-1}-\sum_{i=0}^{n-2}N_iq^i= q^{n+r-1}-\sum_{i=0}^{n+r-2}N_i'q^i,$$
where
\\
$$\begin{aligned}
&N_{n+r-2}'=q-p^m, N_j' = 0 \ \mbox{for}\ n\leq j\leq n+r-3,\\
&N_{n-1}'=N_{n-1}-(q-p^m), N_j'=N_j\  \mbox{for}\ 0\leq j\leq n-2.
\end{aligned}$$
\\
Moreover,
$0\leq\sum_{i=0}^{n+r-2} N_i' = q-p^m+N_{n-1}-(q-p^m)+\sum_{i=0}^{n-2}N_i =\sum_{i=0}^{n-1}N_i \leq q-1$.
By Theorem \ref{mainproposition},
\
\\
$$\begin{aligned}
&H_{s_1-1}=H_{q^n-\sum N_i q^i -1}=B_1\prod_{i=1}^{n-1} (t-\theta^{q^i})^{\sum_{j=0}^{n-1-i} N_jq^j},\\
&H_{w-1}=H_{q^{n+r-1}-\sum N_iq^i-1}=B_2\prod_{i=1}^{n+r-2} (t-\theta^{q^i})^{\sum_{j=0}^{n+r-2-i} N_j'q^j},
\end{aligned}$$
\\
where
$$\begin{aligned}
B_1&=\Gamma_{q^n-\sum N_i q^i}|_{\theta =t} \frac{(-1)^{\sum_{i=0}^{n-2} (n-1-i)N_iq^i}}{\prod_{i=0}^{n-2}L_{n-1-i}|_{\theta =t}^{N_iq^i}},\\
B_2&=\Gamma_{q^{n+r-1}-\sum N_i'q^i}|_{\theta =t}\frac{(-1)^{\sum_{i=0}^{n-1} (n+r-2-i)N_i'q^i}}{\prod_{i=0}^{n-1}L_{n+r-2-i}|_{\theta =t}^{N_i'q^i}}.
\end{aligned}$$
\\
Then
\\
$$\begin{aligned}
H_{s_1-1}^{(-1)}&=B_1\prod_{i=0}^{n-2} (t-\theta^{q^{i}})^{\sum_{j=0}^{n-2-i} N_jq^j},\\
H_{w-1}^{(-r)}&=B_2\prod_{i=1}^{n+r-2} (t-\theta^{q^{i-r}})^{\sum_{j=0}^{n+r-2-i} N_j'q^j},\\
&=B_2\prod_{i=1}^{r-1} (t-\theta^{\frac{1}{q^{i}}})^{\sum_{j=0}^{n-1} N_j'q^j}\prod_{i=0}^{n-2} (t-\theta^{q^{i}})^{\sum_{j=0}^{n-2-i} N_j'q^j}\\
&=\frac{B_2}{B_1}H_{s_1-1}^{(-1)}\prod_{i=1}^{r-1} (t-\theta^{\frac{1}{q^{i}}})^{\sum_{j=0}^{n-1} N_j'q^j}\\
\delta_2^{(-1)}&=f_2|_{\theta =t} [(t-\theta^{\frac{1}{q}})\cdots(t-\theta^{\frac{1}{q^{r-1}}})]^{p^mq^{n+r-2}}.
\end{aligned}$$
\\
Hence
$$\begin{aligned}
&f_1\delta_{2}^{(-1)}H_{s_1-1}^{(-1)}(t-\theta)^{w}+bH_{w-1}^{(-r)}(t-\theta)^w\cdots(t-\theta^{\frac{1}{q^{r-1}}})^w\\
&=(t-\theta)^{w}H_{s_1-1}^{(-1)}[(t-\theta^{\frac{1}{q}})\cdots(t-\theta^{\frac{1}{q^{r-1}}})]^{p^mq^{n+r-2}}[f_1f_2|_{\theta =t}+b\frac{B_2}{B_1}].
\end{aligned}$$
So the equation of $\delta_1$ will be solved if we take $b=-f_2|_{\theta =t}B_1$ and $f_1=B_2$.

Finally, since $q-1$ does not divide $\sum s_i$, we can apply $f_2$ in Remark \ref{f2} and get
\\
$$\begin{aligned}
&\zeta_A(q^n-\sum N_iq^i, p^m(q-1)q^{n-1},\cdots,p^m(q-1)q^{n+r-3})\\
&=\frac{-b(\theta)\Gamma_{q^{n+r-1}-\sum N_i'q^i}}{f_1(\theta)a(\theta)\Gamma_{q^n-\sum N_iq^i}\Gamma_{p^m(q-1)q^{n-1}}\cdots \Gamma_{p^m(q-1)q^{n+r-3}}}\zeta_A(q^{n+r-1}-\sum N_i'q^i)\\
&=\frac{f_2B_1(\theta)\Gamma_{q^{n+r-1}-\sum N_i'q^i}}{B_2(\theta)\Gamma_{q^n-\sum N_iq^i}\Gamma_{p^m(q-1)q^{n-1}}\cdots \Gamma_{p^m(q-1)q^{n+r-3}}(-1)[1]^{p^mq^{n+r-3}}\cdots [r-1]^{p^mq^{n-1}}}\\
&=\frac{(-1)^r \frac{(-1)^{\sum_{i=0}^{n-2} (n-1-i)N_iq^i}}{\prod_{i=0}^{n-2}L_{n-1-i}^{N_iq^i}}}{\frac{(-1)^{\sum_{i=0}^{n-1} (n+r-2-i)N_i'q^i}}{\prod_{i=0}^{n-1}L_{n+r-2-i}^{N_i'q^i}}(-1)[1]^{p^mq^{n+r-3}}\cdots [r-1]^{p^mq^{n-1}}}\\
&=\frac{(-1)^{r+1}(-1)^{\sum_{i=0}^{n-2} (n-1-i)N_iq^i}\prod_{i=0}^{n-1}L_{n+r-2-i}^{N_i'q^i}}{(-1)^{\sum_{i=0}^{n-1} (n+r-2-i)N_i'q^i}\prod_{i=0}^{n-2}L_{n-1-i}^{N_iq^i}[1]^{p^mq^{n+r-3}}\cdots [r-1]^{p^mq^{n-1}}}\\
&=\displaystyle\frac{(-1)^{(r-1)(1+\sum N_i)}L_{n+r-2}^{N_0q^0}\cdots L_{r-1}^{(N_{n-1}-(q-p^m))q^{n-1}}}{[1]^{p^mq^{n+r-3}} [2]^{p^mq^{n+r-4}} \cdots [r-1]^{p^mq^{n-1}}L_{n-1}^{N_0q^0}\cdots L_1^{N_{n-2}q^{n-2}}}\zeta_A(q^{n+r-1}-\sum N_i'q^i).
\end{aligned}$$
\end{proof}\
\\
At the end of this section we prove Conjecture \ref{LKconj} $($b$)$.
\begin{theorem}\label{mainthm4}
For any $q>2$, $\zeta_A(1,q(q-1),q^3-q^2+q-1)$ is zeta-like. Furthermore,
$$\zeta_A(1,q(q-1),q^3-q^2+q-1)= \frac{[3]-1}{[3][2][1]^{q^2-q+1}}\zeta_A(q^3).$$
\end{theorem}
\begin{proof}
We prove this theorem by providing a solution of Equation (\ref{E:Eqdelta}).
Let
\\
$$\begin{aligned}
&a= \{\Gamma_{q^3}[3]\}|_{\theta=t},\\
&b= \{[1]^{q-3}[2]^{q-2}(-[3]+1)\}|_{\theta=t},\\
&\delta_1= \frac{a[2]^{q-2}[1]^{q-3}|_{\theta=t}}{[3]|_{\theta=t}}(t-\theta)^{q^3}\left[(t-\theta^q)(t-\theta^{\frac{1}{q}})^{q^3}-\theta^{q^2}+t+1\right],\\
&\delta_2=\frac{-a[2]^{q-2}[1]^{q-3}|_{\theta=t}}{[3]|_{\theta=t}}(t-\theta)^{q^3}(t-\theta^{\frac{1}{q}})^{q^3}(t-\theta^q),\\
&\delta_3=\frac{a[2]^{q-2}|_{\theta=t}}{[1]|_{\theta=t}}(t-\theta)^{q^3}(t-\theta^q) .
\end{aligned}$$
\\
Then by Proposition \ref{example} and Theorem \ref{mainproposition},
\\
\\
$(1)$
\\
\\
$\begin{aligned}
&[\delta_3^{(-1)}+aH_{q^3-q^2+q-2}^{(-1)}](t-\theta)^{q^3-q^2+q-1}\\
&=\frac{a[2]^{q-2}|_{\theta=t}}{[1]|_{\theta=t}}(t-\theta)^{q^3-q^2+q}\left[(t-\theta^{\frac{1}{q}})^{q^3}-[1]|_{\theta=t}\{(t-\theta)^{q^2-q}+[1]|_{\theta=t}^{q^2-1}\}\right]\\
&=\frac{a[2]^{q-2}|_{\theta=t}}{[1]|_{\theta=t}}(t-\theta)^{q^3-q^2+q}\left[t^{q^3}-\theta^{q^2}-(t^q-t)(t-\theta)^{q^2-q}-(t^{q^3}-t^{q^2})\right]\\
&=\frac{a[2]^{q-2}|_{\theta=t}}{[1]|_{\theta=t}}(t-\theta)^{q^3-q^2+q}(t-\theta)^{q^2-q}(t-\theta^q)= \delta_3.
\end{aligned}$
\\
\\
$(2)$
\\
\\
$\begin{aligned}
&[\delta_2^{(-1)}+\delta_3^{(-1)}H_{q^2-q-1}^{(-1)}](t-\theta)^{q^3-1}\\
&=\frac{-a[2]^{q-2}[1]^{q-3}|_{\theta=t}}{[3]|_{\theta=t}}(t-\theta^{\frac{1}{q}})^{q^3}(t-\theta)(t-\theta)^{q^3-1}\left[(t-\theta^{\frac{1}{q^2}})^{q^3}-
\frac{[3]\Gamma_{q^2-q}|_{\theta=t}}{[1]|_{\theta=t}^{q-2}}\right]\\
&=\frac{-a[2]^{q-2}[1]^{q-3}|_{\theta=t}}{[3]|_{\theta=t}}(t-\theta^{\frac{1}{q}})^{q^3}(t-\theta)(t-\theta)^{q^3-1}(t-\theta^q)= \delta_2.
\end{aligned}$
\\
\\
$(3)$
\\
\\
$\begin{aligned}
&[\delta_1^{(-1)}+\delta_2^{(-1)}H_{0}^{(-1)}+bH_{q^3-1}^{(-1)}](t-\theta)^{q^3}\\
&=(t-\theta)^{q^3}\left[\frac{a[2]^{q-2}[1]^{q-3}|_{\theta=t}}{[3]|_{\theta=t}}(t-\theta^{\frac{1}{q}})^{q^3}(-\theta^{q}+t+1)+\{\Gamma_{q^3}[1]^{q-3}[2]^{q-2}(-[3]+1)\}|_{\theta=t}\right]\\
&=\frac{a[2]^{q-2}[1]^{q-3}|_{\theta=t}}{[3]|_{\theta=t}}(t-\theta)^{q^3}\left[(t^{q^3}-\theta^{q^2})(-\theta^{q}+t+1)-t^{q^3}+t+1\right]= \delta_1.
\end{aligned}$
\\
Since $q>2$, the ratio of $\zeta_A(1,q(q-1),q^3-q^2+q-1)$ to $\zeta_A(q^3)$ is
$$\frac{-b(\theta)\Gamma_{q^3}}{a(\theta) \Gamma_{1}\Gamma_{q^2-q}\Gamma_{q^3-q^2+q-1}}= \frac{[3]-1}{[3][2][1]^{q^2-q+1}}.$$
\end{proof}

\section{Main result on Eulerian multizeta values} \label{Eulerian}
In this section we will present two families of Eulerian multizeta values mentioned in Conjecture \ref{LRTconj} $(a)$ and Conjecture \ref{LRTconj} $(b)$ for $q=2$.  As a consequence, this confirms that all the multizeta values conjectured to be Eulerian in \cite{CPY}, Section 6.2, are indeed Eulerian.
\begin{theorem}\label{mainthm1}
For any positive integer $r>1$ and $n$, we have

$$\begin{aligned}
&\  \zeta_{A}(q^n-1, (q-1)q^n, \cdots, (q-1)q^{n+r-2} )\\
&= \zeta_{A}(q^n-1)\zeta_{A}((q-1), \cdots, (q-1)q^{r-2} )^{q^n}-\zeta_{A}(q^{n+1}-1, (q-1)q^{n+1}, \cdots, (q-1)q^{n+r-2} ).\\
\end{aligned}$$
Moreover, one derives from this formula that
$$\begin{aligned}
\zeta_{A}(q^n-1, (q-1)q^n, \cdots, (q-1)q^{n+r-2} )&= \displaystyle\frac{(-1)^{n+r-1}[n+r-2][n+r-3]\cdots[n]}{[1]^{q^{n+r-2}}[2]^{q^{n+r-3}}\cdots[r-1]^{q^{n}}L_{n+r-1}}\widetilde{\pi}^{q^{n+r-1}-1}\\
&=\displaystyle\frac{[n+r-2][n+r-3]\cdots[n]}{[1]^{q^{n+r-2}}[2]^{q^{n+r-3}}\cdots[r-1]^{q^{n}}}\zeta_{A}(q^{n+r-1}-1)
\end{aligned}$$

\end{theorem}

\begin{theorem}\label{mainthm2}
If $q=2$, then we have $$\zeta_A(1)\zeta_A(1,2,\cdots,2^{r-1})= \zeta_A(1,3,2^2,\cdots,2^{r-1}) + \zeta_A (1,1,2,\cdots,2^{r-2})^2$$ for $r>1$.
Furthermore, we have
$$\begin{aligned}
\zeta_A(1,3,2^2,\cdots,2^{r-1})&= \displaystyle \frac{[r]-1}{L_1[r]}\frac{1}{L_1^{2^{r-1}}L_2^{2^{r-2}}\cdots L_{r-3}^{2^2}L_{r-2}^{2^2}L_1^{2^r}}\widetilde{\pi}^{2^r}\\
&=\displaystyle \frac{[r]-1}{L_1[r]}\frac{1}{L_1^{2^{r-1}}L_2^{2^{r-2}}\cdots L_{r-3}^{2^2}L_{r-2}^{2^2}}\zeta(2^r).
\end{aligned}$$
for $r>1$.
\end{theorem}

Aside from Carlitz's evaluations in Lemma \ref{bernoulli}, the key point of the proof is the relations among the power sums $S_d(m)$.

\begin{lemma}\label{main lemma1}
For any $d\geq 1$, one has
$$S_d(q^n-1)S_d((q-1)q^n) = S_d(q^{n+1}-1)- S_d ((q-1)q^n, q^n-1)$$
\end{lemma}
\begin{proof}
See \cite[pp.2332]{T1}.
\end{proof}

\begin{proof1}\ \ref{mainthm1}
\\

Consider the product
$$\begin{aligned}
&\ \zeta_{A}(q^n-1)\zeta_{A}((q-1)q^n, \cdots, (q-1)q^{n+r-2} )\\
&= \sum_{d\geq0}S_d(q^n-1) \sum_{d_1>d_2>\cdots>d_{r-1}\geq 0} S_{d_1}((q-1)q^n)\cdots S_{d_{r-1}}((q-1)q^{n+r-2})\\
&=\sum_{d_1>d_2>\cdots>d_{r-1}\geq 0\atop d>d_1} + \sum_{d_1>d_2>\cdots>d_{r-1}\geq 0\atop d=d_1} +\sum_{d_1>d_2>\cdots>d_{r-1}\geq 0\atop d<d_1}
S_d(q^n-1) S_{d_1}((q-1)q^n)\cdots S_{d_{r-1}}((q-1)q^{n+r-2}) \\
&\hspace{450pt}(\ref{mainthm1}.1)
\end{aligned}$$
By Lemma \ref{main lemma1},
$$\begin{aligned}
&= \sum_{d_1>d_2>\cdots>d_{r-1}\geq 0\atop d=d_1} S_d(q^n-1) S_{d_1}((q-1)q^n)\cdots S_{d_{r-1}}((q-1)q^{n+r-2})\\
&= \sum_{d_1>d_2>\cdots>d_{r-1}\geq 0}  [S_{d_1}(q^{n+1}-1)- S_{d_1} ((q-1)q^n, q^n-1)]S_{d_2}((q-1)q^{n+1})\cdots S_{d_{r-1}}((q-1)q^{n+r-2})\\
&=\sum_{d_1>d_2>\cdots>d_{r-1}\geq 0}  S_{d_1}(q^{n+1}-1)S_{d_2}((q-1)q^{n+1})\cdots S_{d_{r-1}}((q-1)q^{n+r-2})\\
&- \sum_{d_1>d_2>\cdots>d_{r-1}\geq 0 \atop d'<d_1} S_{d_1} ((q-1)q^n)S_{d'}(q^n-1)S_{d_2}((q-1)q^{n+1})\cdots S_{d_{r-1}}((q-1)q^{n+r-2})
\end{aligned}$$
So (\ref{mainthm1}.1) leads to

$$\begin{aligned}
&\sum_{d_1>d_2>\cdots>d_{r-1}\geq 0\atop d>d_1}S_d(q^n-1) S_{d_1}((q-1)q^n)\cdots S_{d_{r-1}}((q-1)q^{n+r-2})\\
&+ \sum_{d_1>d_2>\cdots>d_{r-1}\geq 0} S_{d_1}(q^{n+1}-1)S_{d_2}((q-1)q^{n+1})\cdots S_{d_{r-1}}((q-1)q^{n+r-2})\\
&= \zeta_{A}(q^n-1, (q-1)q^n, \cdots, (q-1)q^{n+r-2} )+ \zeta_{A}(q^{n+1}-1, (q-1)q^{n+1}, \cdots, (q-1)q^{n+r-2} )
\end{aligned}$$
\\
For the second part, we will prove by mathematical induction on $r>1$.
When $r=2$,   $$\zeta_{A}(q^n-1, (q-1)q^n) = \zeta_{A}(q^n-1)\zeta_{A}((q-1)q^n) -\zeta_{A}(q^{n+1}-1)$$
Since $\zeta_{A}(q^n-1) = \displaystyle\frac{(-1)^n}{L_n}\widetilde{\pi}^{q^{n}-1} $, we have
\\
\\
$\begin{aligned}
&\ \ \ \zeta_{A}(q^n-1)\zeta_{A}((q-1)q^n) -\zeta_{A}(q^{n+1}-1)\\
&=\displaystyle\frac{(-1)^n}{L_n}\widetilde{\pi}^{q^{n}-1}(\displaystyle\frac{(-1)}{L_1}
\widetilde{\pi}^{q-1})^{q^n}-\displaystyle\frac{(-1)^{n+1}}{L_{n+1}}\widetilde{\pi}^{q^{n+1}-1}\\
&=[\displaystyle\frac{(-1)^{n+q^n}}{L_nL_1^{q^n}}-\displaystyle\frac{(-1)^{n+1}}{L_{n+1}}]\widetilde{\pi}^{q^{n+1}-1}\\
&=\displaystyle\frac{(-1)^{n+1}[[n+1]-[1]^{q^n}]}{L_{n+1}L_1^{q^n}}\widetilde{\pi}^{q^{n+1}-1}\\
&=\displaystyle\frac{(-1)^{n+1}[n]}{L_{n+1}L_1^{q^n}}\widetilde{\pi}^{q^{n+1}-1} = \displaystyle\frac{[n]}{[1]^{q^n}}\zeta(q^{n+1}-1).
\end{aligned}$
\\
\\
Assume that the statement holds for any $n$ with depth $<r$, then by the first recursive formula and induction hypothesis we have
\\
\\
$\begin{aligned}
&\ \ \zeta_{A}(q^n-1, (q-1)q^n, \cdots, (q-1)q^{n+r-2} )\\
&= \zeta_{A}(q^n-1)\zeta_{A}((q-1), \cdots, (q-1)q^{r-2} )^{q^n}-\zeta_{A}(q^{n+1}-1, (q-1)q^{n+1}, \cdots, (q-1)q^{n+r-2} ).\\
&= \displaystyle\frac{(-1)^n}{L_n}\widetilde{\pi}^{q^{n}-1}(\displaystyle\frac{(-1)^{r-1}[r-2]\cdots[1]}{[1]^{q^{r-2}}\cdots[r-2]^{q}L_{r-1}}\widetilde{\pi}^{q^{r-1}-1})^{q^n}
-  \displaystyle\frac{(-1)^{n+r-1}[n+r-2]\cdots[n+1]}{[1]^{q^{n+r-2}}\cdots[r-2]^{q^{n+1}}L_{n+r-1}}\widetilde{\pi}^{q^{n+r-1}-1}\\
&=[\displaystyle\frac{(-1)^{n+(r-1)q^n}}{[1]^{q^{n+r-2}}\cdots[r-2]^{q^{n+1}}[r-1]^{q^{n}}L_{n}}-\displaystyle\frac{(-1)^{n+r-1}[n+r-2]\cdots[n+1]}{[1]^{q^{n+r-2}}\cdots[r-2]^{q^{n+1}}L_{n+r-1}}]\widetilde{\pi}^{q^{n+r-1}-1}\\
&=[\displaystyle\frac{(-1)^{n+r-1}[n+r-1]\cdots[n+1]-(-1)^{n+r-1}[r-1]^{q^n}[n+r-2]\cdots[n+1]}{[1]^{q^{n+r-2}}\cdots[r-2]^{q^{n+1}]}[r-1]^{q^{n}}L_{n+r-1}}]\widetilde{\pi}^{q^{n+r-1}-1}\\
&=[\displaystyle\frac{(-1)^{n+r-1}[n+r-2]\cdots[n+1][n]}{[1]^{q^{n+r-2}}\cdots[r-2]^{q^{n+1}]}[r-1]^{q^{n}}L_{n+r-1}}]\widetilde{\pi}^{q^{n+r-1}-1}\\
&=\displaystyle\frac{[n+r-2][n+r-3]\cdots[n]}{[1]^{q^{n+r-2}}[2]^{q^{n+r-3}}\cdots[r-1]^{q^{n}}}\zeta_{A}(q^{n+r-1}-1).
\end{aligned}$
\\
\\

\end{proof1}
\
\\
Similarly, by Carlitz's evaluations and  relations on power sums we can prove Theorem \ref{mainthm2}.
\
\\
\\
\begin{proof1}\ \ref{mainthm2}

Observe that
$$\begin{aligned}
\zeta_A(1,1,2,\cdots,2^{r-1})&=\sum_{d_1>d_2>\cdots>d_{r}\geq 0\atop d>d_1} S_d(1) S_{d_1}(1)\cdots S_{d_{r}}(2^{r-1}) \\
&=\sum_{d_1>d_2>\cdots>d_{r}\geq 0\atop d_2<d<d_1}S_d(1) S_{d_1}(1)\cdots S_{d_{r}}(2^{r-1}).
\end{aligned}$$

$$\begin{aligned}
&\ \zeta_A(1)\zeta_A(1,2,\cdots,2^{r-1})\\
&= \sum_{d\geq0}S_d(1) \sum_{d_1>d_2>\cdots>d_{r}\geq 0} S_{d_1}(1)\cdots S_{d_r}(2^{r-1})\\
&=\sum_{d_1>d_2>\cdots>d_{r}\geq 0\atop d>d_1} + \sum_{d_1>d_2>\cdots>d_{r}\geq 0\atop d=d_1} +\sum_{d_1>d_2>\cdots >d_{r}\geq 0\atop d_2<d<d_1}
+\sum_{d_1>d_2>\cdots>d_{r}\geq 0\atop d=d_2}+\sum_{d_1>d_2>\cdots>d_{r}\geq 0\atop d<d_2}S_d(1) S_{d_1}(1)\cdots S_{d_{r}}(2^{r-1}) \\
&=\sum_{d_1>d_2>\cdots>d_{r}\geq 0\atop d=d_1}
+\sum_{d_1>d_2>\cdots>d_{r}\geq 0\atop d=d_2}+\sum_{d_1>d_2>\cdots>d_{r}\geq 0\atop d<d_2}S_d(1) S_{d_1}(1)\cdots S_{d_{r}}(2^{r-1}), \\
&\hspace{450pt}(\ref{mainthm2}.1)
\end{aligned}$$
\\
For $d=d_1$, it is clear that $$S_{d_1}(1)S_{d_1}(1) = S_{d_1}(2).$$
\\
For $d=d_2$, by using $(\ref{main lemma1}.1)$ we have $$S_{d_2}(1)S_{d_2}(2) = S_{d_2}(3)+ S_{d_2}(2,1).$$
\\
It follows that (\ref{mainthm2}.1) equals

$\begin{aligned}
&\sum_{d_1>d_2>\cdots>d_{r}\geq 0\atop d=d_1} S_{d_1}(2)S_{d_2}(2)\cdots S_{d_{r}}(2^{r-1}) + \sum_{d_1>d_2>\cdots>d_{r}\geq 0\atop d=d_2}
S_{d_1}(1)[S_{d_2}(3)+S_{d_2}(2,1)]\cdots S_{d_{r}}(2^{r-1})\\
&+\sum_{d_1>d_2>\cdots>d_{r}\geq 0\atop d<d_2}S_d(1) S_{d_1}(1)\cdots S_{d_{r}}(2^{r-1})\\
&= \sum_{d_1>d_2>\cdots>d_{r}\geq 0\atop d=d_1} S_{d_1}(2)S_{d_2}(2)\cdots S_{d_{r}}(2^{r-1}) + \sum_{d_1>d_2>\cdots>d_{r}\geq 0\atop d=d_2}
S_{d_1}(1)S_{d_2}(3)\cdots S_{d_{r}}(2^{r-1})\\
&+2\sum_{d_1>d_2>\cdots>d_{r}\geq 0\atop d<d_2}S_d(1) S_{d_1}(1)\cdots S_{d_{r}}(2^{r-1})\\
&=\zeta_A(2,2,2^2,\cdots,2^{r-1}) + \zeta_A(1,3,2^2,\cdots,2^{r-1})\\
&=\zeta_A(1,1,2,\cdots,2^{r-2})^2 + \zeta_A(1,3,2^2,\cdots,2^{r-1}).
\end{aligned} $
\\

For the second part,
by Lemma \ref{bernoulli}, Theorem \ref{LRTthm1} and \ref{mainthm1}, we have
$$\begin{aligned}
&\zeta_A(1) = \displaystyle\frac{1}{L_1} \widetilde{\pi}\\
&\zeta_A(1,1,2,\cdots,2^{r-2}) =  \displaystyle\frac{1}{[1]^{2^{r-2}}[2]^{2^{r-3}}\cdots [r-1]^{2^0}L_1^{2^{r-1}}}\widetilde{\pi}^{2^{r-1}}\\
&\zeta_A(1,2,\cdots,2^{r-1}) = \displaystyle\frac{[r-1]\cdots[1]}{[1]^{2^{r-1}}\cdots[r-1]^2L_r}\widetilde{\pi}^{2^r-1}
\end{aligned}$$
\
\\
So
\
\\
\\
$\begin{aligned}
&\  \zeta_A(1,3,2^2,\cdots,2^{r-1})\\
&=\zeta_A(1)\zeta_A(1,2,\cdots,2^{r-1}) + \zeta_A(1,1,2,\cdots,2^{r-2})^2\\
&= \displaystyle\frac{[r-1]\cdots[1]}{[1]^{2^{r-1}}\cdots[r-1]^2L_rL_1}\widetilde{\pi}^{2^r}+
\displaystyle\frac{1}{[1]^{2^{r-1}}[2]^{2^{r-2}}\cdots [r-1]^{2^1}L_1^{2^{r}}}\widetilde{\pi}^{2^{r}}\\
&=\displaystyle\frac{1}{[1]^{2^{r-1}}[2]^{2^{r-2}}\cdots [r-1]^{2^1}}\widetilde{\pi}^{2^{r}}[\displaystyle\frac{L_{r-1}}{L_rL_1}+\displaystyle\frac{1}{L_1^{2^r}}]\\
&=\displaystyle\frac{1}{[1]^{2^{r-1}}[2]^{2^{r-2}}\cdots [r-1]^{2^1}}\displaystyle\frac{L_1^{2^r-1}+[r]}{[r]}\displaystyle\frac{1}{L_1^{2^r}}\widetilde{\pi}^{2^{r}}\\
&\hspace{400pt} (\ref{mainthm2}.1)
\end{aligned}$
\
\
\\
\\
Note that

$$\begin{aligned}
L_1^{2^{r-2}}L_2^{2^{r-3}}\cdots L_{r-3}^{2^2}L_{r-2}^{2^2} &=
[1]^{2^{r-2}+2^{r-3}+\cdots+2^2+2^2}[2]^{2^{r-3}+2^{r-4}+\cdots+2^2+2^2}\cdots[r-2]^{2^2}\\
&=[1]^{2^{r-1}}[2]^{2^{r-2}}\cdots[r-2]^{2^2}.
\end{aligned}$$
\\
Here we adopt the usual convention on empty sums, products and indexing.

On the other hand we have

$$\begin{aligned}
\displaystyle\frac{L_1^{2^r-1}+[r]}{[r][r-1]^2} &= \displaystyle\frac{L_1^{2^r}+[r][1]}{[r][r-1]^2[1]}\\
&= \displaystyle\frac{[r+1]+[r]+[r][1]}{[r][r-1]^2[1]}\\
&= \displaystyle\frac{[r-1]^2([r]+1)}{[r][r-1]^2[1]} = \displaystyle\frac{[r]+1}{[r][1]}.
\end{aligned}$$
\
\\
Applying this to (\ref{mainthm2}.1), we finish the proof.
\end{proof1}
\
\\
{\bf Acknowledgments} The author is grateful to J. Yu for his careful reading, useful discussions and helpful suggestions. She also thanks
C.-Y. Chang for helpful comments, and thanks Y.-L. Kuan and Y.-S. Lin for their support of the data collection.

\end{document}